\documentclass[12pt,dvipsnames]{amsart}

\usepackage[pdftex,colorlinks=true,urlcolor=blue,linkcolor=blue,citecolor=blue]{hyperref}
\usepackage{amssymb}
\usepackage{tikz-cd}
\usepackage[capitalize]{cleveref}
\usepackage{stmaryrd}

\usepackage{mathrsfs}
\usepackage{fullpage}
\usepackage{enumitem}

\theoremstyle{plain}
\newtheorem{Theorem}{Theorem}[section]

\newtheorem{Lemma}[Theorem]{Lemma}
\newtheorem{lemma}[Theorem]{Lemma}
\newtheorem{Proposition}[Theorem]{Proposition}
\newtheorem{Corollary}[Theorem]{Corollary}

\newtheorem{Question}[Theorem]{Question}
\newtheorem{Claim}{Claim}
\theoremstyle{definition}
\newtheorem{Definition}[Theorem]{Definition}
\newtheorem{Remark}[Theorem]{Remark}

\newtheorem{example}[Theorem]{Example}

\numberwithin{equation}{section}

\newcommand{\Erdos}{Erd\H{o}s}
\newcommand{\Folner}{F\o{}lner}

\newcommand{\Szemeredi}{Szemer\'{e}di}

\newcommand{\Turan}{Tur{\'a}n}

\newcommand{\omg}{\omega}

\newcommand{\N}{\mathbb{N}}
\newcommand{\Z}{\mathbb{Z}}
\newcommand{\R}{\mathbb{R}}
\newcommand{\C}{\mathbb{C}}
\newcommand{\Q}{\mathbb{Q}}
\newcommand{\T}{\mathbb{T}}

\newcommand{\inv}{\mathcal{I}}

\newcommand{\cont}{C}%{\mathsf{C}}
\newcommand{\ultra}[1]{\mathsf{#1}}
\newcommand{\lp}{L}%{\mathsf{L}}
\newcommand{\orbit}{\mathcal{O}}
\newcommand{\meas}{\mathcal{M}}

\newcommand{\ball}{\mathsf{B}}

\newcommand{\define}[1]{{\itshape #1}}

\renewcommand{\leq}{\leqslant}
\renewcommand{\geq}{\geqslant}
\renewcommand{\setminus}{\backslash}

\newcommand{\E}{\mathbb{E}}

\newcommand{\one}{\boldsymbol{1}}

\newcommand{\condex}{\mathbb{E}}

\newcommand{\gen}{{\mathsf{gen}}}
\newcommand{\supp}{{\mathsf{supp}}}

\newcommand{\kone}{{\llbracket 1 \rrbracket}}
\newcommand{\two}{{\llbracket 2 \rrbracket}}
\newcommand{\three}{{\llbracket 3 \rrbracket}}
\renewcommand{\k}{{\llbracket k \rrbracket}}
\newcommand{\kup}{{\llbracket k+1 \rrbracket}}
\newcommand{\kdown}{{\llbracket k-1 \rrbracket}}

\renewcommand{\d}{\,\mathsf{d}}

\newcommand{\intd}{\,\mathsf{d}}

\newcommand{\ghk}{|\!|\!|}
\newcommand{\cube}{\mathsf{Q}}
\newcommand{\erd}{\mathsf{E}}
\newcommand{\nilc}{\mathsf{N}}

\author[B. Kra]{Bryna Kra}
\address{Northwestern University, Evanston, Illinois, USA}
\thanks{B. Kra  acknowledges National Science Foundation grant DMS-205464}
\author[J. Moreira]{Joel Moreira}
\address{University of Warwick, Coventry, UK}
\author[F. K. Richter]{Florian K. Richter}
\address{EPFL, Lausanne, Switzerland}
\author[D. Robertson]{Donald Robertson}
\address{University of Manchester, Manchester, UK}
\thanks{D. Robertson acknowledges EPSRC grant V050362}

\title{Infinite Sumsets in Sets with Positive Density}

\begin{document}

\begin{abstract} 
Motivated by questions asked by \Erdos{}, we prove that any set $A\subset\N$ with positive upper density contains, for any $k\in\N$, a sumset $B_1+\cdots+B_k$, where $B_1,\dots,B_k\subset\N$ are infinite.
Our proof uses ergodic theory and relies on structural results for measure preserving systems. Our techniques are new, even for the previously known case of $k=2$.
\end{abstract}

\maketitle

\tableofcontents

\section{Introduction}
\label{sec_intro}

\Szemeredi{}~\cite{S}, settling a longstanding conjecture of \Erdos{} and \Turan{}, proved that every set of natural numbers with positive upper density contains arbitrarily long arithmetic progressions.
Since that work, a variety of generalizations have appeared, guaranteeing the existence of various finite patterns in sufficiently large sets of integers. 
Amongst the notable achievements are the multidimensional generalization of Furstenberg and Katznelson~\cite{Furstenberg-Katznelson}, the polynomial generalization of Bergelson and Leibman~\cite{Bergelson-Leibman}, and the groundbreaking theorem of Green and Tao~\cite{Green_Tao08} on arithmetic progressions in primes.

\Erdos{}~\cite{EG} conjectured that sets of natural numbers with positive upper density contain not only finite arithmetic configurations, but various infinite ones as well.
In particular, one of his conjectures stated that any set with positive upper density contains a sumset $B_1+B_2=\{b_1+b_2: b_1\in B_1,~b_2\in B_2\}$ of two infinite sets $B_1,B_2\subset\N$. Under the stronger assumption that the density exceeds $1/2$, this conjecture was proven in~\cite{DGJLLM15}. In full generality, it was resolved in~\cite{MRR}, using 
a combination of combinatorial analysis, nonstandard analysis, and an ergodic structural result.
A more streamlined and purely ergodic proof was subsequently given by Host in~\cite{Host}.

We are concerned here with a natural generalization of this result to higher orders.
To be precise, we focus on the question of whether sets with positive upper density must contain for each $k \in \N$ a sumset
\[
B_1 + \cdots + B_k = \{ b_1 + \cdots + b_k : b_1 \in B_1,\dots,b_k \in B_k \}
\]
where $B_1,\dots,B_k \subset \N$ are infinite.
Our main result gives a positive answer to this question under the more general assumption of positive upper Banach density, which we phrase in terms of \Folner{} sequences.
Recall that a \define{\Folner{} sequence} in $\N$ is any sequence $\Phi=(\Phi_N)_{N \in \N}$ of finite subsets of $\N$ 
satisfying 
\begin{equation}
\label{eqn:folner}
\lim_{N \to \infty} \dfrac{|(\Phi_N - t) \cap \Phi_N|}{|\Phi_N|} = 1
\end{equation}
for all $t \in \N$, where $\Phi_N-t$ denotes the shift $\{x\in\N:x+t\in\Phi_N\}$.
For example
\[
\Phi_N = \{ L_N, L_N + 1,\dots, M_N - 1 \}
\]
defines a \Folner{} sequence whenever $M_N - L_N \to \infty$ as $N\to\infty$.
A set $A \subset \N$ has \define{positive upper Banach density} if there is a \Folner{} sequence $\Phi$ such that
\begin{equation}
\label{eqn:main_hyp}
\lim_{N \to \infty} \dfrac{|A \cap \Phi_N|}{|\Phi_N|} > 0.
\end{equation}
With this definition, we are ready to state our main theorem, which answers \cite[Question 6.3]{MRR}.

\begin{Theorem}
\label{thm:main_theorem}
If $A \subset \N$ has positive upper Banach density then for every $k \in \N$ there are infinite sets $B_1,\dots,B_k \subset \N$ with $B_1 + \cdots + B_k \subset A$.
\end{Theorem}

We believe that this result and its antecedents \cite{Nathanson80,DGJLLM15,MRR,Host}, are the first to guarantee infinite arithmetic configurations in all sets of positive upper Banach density.
The latter papers \cite{DGJLLM15,MRR,Host} all relied on a method introduced in~\cite{DGJLLM15} (see~\cite[Proof of Theorem 3.2]{DGJLLM15} and \cite[Proposition 2.5]{MRR}) to construct two-fold sumsets $B_1+B_2$ within positive density sets using a variant of an  \emph{intersectivity lemma} of Bergelson \cite[Theorem 1.1]{bergelson}.
This technique played a particularly important role in \cite{MRR,Host}, where it was used to transfer the combinatorial problem of finding sumsets in sets of positive density into an ergodic theoretic problem of set recurrence in measure preserving dynamical systems.
Although this approach works well when dealing with two-fold sumsets $B_1+B_2$, it is not clear how to extend it to find $B_1 + \cdots + B_k$ in sets of positive density when
$k\geq 3$.
For this reason, it is not clear how to generalize the proofs in \cite{MRR,Host} to produce a proof of \cref{thm:main_theorem}.
We overcome this obstacle by introducing a new approach that is distinct from existing methods. The remainder of this introduction is dedicated to outlining this approach and summarizing its key elements.

The central idea behind our proof of \cref{thm:main_theorem} is to connect $k$-fold sumsets in the integers with points in certain $2^k$-fold joinings of measure preserving systems that have highly structured recurrence properties. 
These joinings -- known as cube systems -- were introduced in~\cite{HK-05} to give an algebraic characterization of the limiting behavior of the multiple ergodic averages introduced by Furstenberg~\cite{Furstenberg-1977} in his proof of \Szemeredi's{} theorem via ergodic theory. Moreover, they are used to define the dynamical counterpart of the Gowers uniformity norms, making them indispensable tools for studying higher order recurrence in dynamical systems.

The density assumption~\eqref{eqn:main_hyp} provides, via a classical construction of Furstenberg, a dynamical system $(X,\mu,T)$, a clopen set $E\subset X$ satisfying $
\mu(E)>0$, and a point $a\in X$ generic for $\mu$ such that 
\begin{equation}
\label{eqn_fc_1}
A=\{n\in\N: T^na\in E\}.
\end{equation}
Writing $\k = \{0,1\}^k$, we consider the $2^k$-fold self-product $X^\k$ equipped with the $2^k$-fold product transformation $T^\k$.
We connect $k$-fold sumsets to certain special points in $X^\k$. 
We refer to such points as \define{$k$-dimensional \Erdos{} cubes}, and they form the bridge between additive combinatorics and ergodic theory in our work.
In the case $k=2$ the definition is relatively easy to state.

\begin{Definition}
\label{def_erdos_cubes_2d}
A point $(x_{00},x_{01},x_{10},x_{11})\in X^\two$ is a \define{$2$-dimensional \Erdos{} cube} if, with respect to the transformation $T\times T$, the forward orbit of $(x_{00},x_{01})$ visits every neighborhood of $(x_{10},x_{11})$ infinitely often, and the forward orbit of $(x_{00},x_{10})$ visits every neighborhood of $(x_{01},x_{11})$ infinitely often.
\end{Definition}

For $k\geq 3$ the definition becomes more technical, requiring additional notation, and is deferred to \cref{def:erdos_cubes} in \cref{sec:erdos-cubes}. 
If $(x_{00},x_{01},x_{10},x_{11})$ is a $2$-dimensional \Erdos{} cube, it can be shown that for any neighborhood $U$ of $x_{11}$ the set $\{n\in\N: T^nx_{00}\in U\}$ contains $B_1+B_2$ for infinite sets $B_1,B_2\subset\N$.
Thus, to prove \cref{thm:main_theorem} for $k=2$ it suffices to find a $2$-dimensional \Erdos{} cube $(x_{00},x_{01},x_{10},x_{11})$ with $x_{00}=a$ and $x_{11}\in E$, where $a$ and $E$ are as in the above construction of Furstenberg,
as then $A$ contains $B_1+B_2$ for infinite sets $B_1,B_2\subset\N$ by \eqref{eqn_fc_1}.
This approach leads to a new proof of \Erdos{}'s conjecture, different from those given  either in~\cite{MRR} or in~\cite{Host}.

For larger $k$, the story is the same: if $(x_{\vec 0},\ldots,x_{\vec 1})\in X^\k$ is a $k$-dimensional \Erdos{} cube then for any neighborhood $U$ of $x_{\vec 1}$ the set $\{n\in\N: T^nx_{\vec 0}\in U\}$ contains $B_1+\cdots+B_k$ for infinite sets $B_1,\ldots,B_k\subset\N$.
The main idea behind our proof of \cref{thm:main_theorem} is to build for every $k \in \N$ a $k$-dimensional \Erdos{} cube with prescribed first and last coordinates.
In realizing it, the theory of cube systems plays an important role.

Cube systems, which we review in \cref{sec:cubism},  enjoy several auspicious properties we exploit in our proof. 
Foremost, it provides for each $k \in \N$ a $2^k$-fold joining $\mu^\k$ of $\mu$ called the \textit{$k$-dimensional cubic measure}, which possesses exceptional symmetry: there is a natural group of symmetries acting on $X^\k$ that fixes the first coordinate of $X^\k$, commutes with $T^\k$, and preserve the measure $\mu^\k$.
Using these symmetries it is not hard to prove that $\mu^\k$-almost every point is an \Erdos{} cube.
To additionally prescribe the first and last coordinates, we use a disintegration $t\mapsto \sigma^\k_t$ of the cubic measure $\mu^\k$ over the projection from $X^\k$ to its first coordinate.
From this perspective, if $\sigma^\k_a$-almost every point is an \Erdos{} cube then we have an abundance of \Erdos{} cubes with  first coordinate equal to  $a$, and this would suffice for proving  \cref{thm:main_theorem}. 

Unfortunately, the general theory of disintegrations only provides a measure $\sigma^\k_t$ for $\mu$-almost every point $t \in X$, and so a priori there is no guarantee that $\sigma_a^\k$ is well defined for a given point $a\in X$.
By passing to an extension of the measure preserving system $(X, \mu, T)$ we can assume that the factor maps defining certain structural factors, the \define{pronilfactors}, are not just measurable, but are actually continuous.
This allows us to make use of topological properties of these factors.  
The bulk of our technical work is showing that this additional topological input implies the existence of an ergodic decomposition $x\mapsto \lambda_x^\k$ of the cubic measure $\mu^\k$ that is not just measurable but actually continuous. 
We use this decomposition to define $\sigma^\k_t$ for \textit{every} point $t \in X$ in such a way that $\sigma^\k_t$ is invariant under the above-mentioned symmetries of $\mu^\k$ and pushes forward to $\mu$ on the last coordinate.
This allows us to prove that $\sigma^\k_a$-almost every point is an \Erdos{} cube whenever $a$ is generic for $\mu$.
The existence of a $k$-dimensional \Erdos{} cube $(x_{\vec 0},\ldots,x_{\vec 1})$ with $x_{\vec 0}=a$ and $x_{\vec 1}\in E$ now follows readily, yielding a proof of \cref{thm:main_theorem}.

The relevance of cube systems to finding sumsets in positive density sets highlights some surprising parallels between dynamical approaches to \Szemeredi{}'s theorem and our approach to proving \cref{thm:main_theorem}: finding $(k+1)$-term arithmetic progressions and finding $k$-fold sumsets in a set of positive upper density rely on the same algebraic structures provided by the $k$-step pronilfactors that arise in the structure theory of measure preserving systems.

Section~\ref{sec:preliminaries} is a preparatory section that covers preliminaries from ergodic theory needed in the proof of \cref{thm:main_theorem}. Depending on the reader's background in ergodic theory, this section may be skimmed through or skipped.

In Section~\ref{sec_B+C}, we carry out our proof of \cref{thm:main_theorem} in the special case $k = 2$. We treat this case separately to illustrate our method for the approach to the general case. 
The remainder of the paper is devoted to the generalization of this proof to higher orders.

In \cref{sec:erdos-cubes} we formulate our main technical result, \cref{thm:cubes_exist}, which states that $k$-dimensional \Erdos{} cubes with first coordinate prescribed and last coordinate in any fixed set of positive measure exist.
The deduction of \cref{thm:main_theorem} from \cref{thm:cubes_exist} is carried out in \cref{sec:building_sumsets}.

In \cref{sec:cubism} we recall the theory of cubic measures and pronilfactors, and prove every measure preserving system has an extension with continuous factor maps to its pronilfactors.
We also state \cref{thm:cubes_exist_pronil} and show how it implies \cref{thm:cubes_exist}.
In \cref{subsec:comparing_cubes}, we also explain the relation between our new notion of \Erdos{} cubes and the already established notion of dynamical cubes introduced in~\cite{HK-05, HKM}.

The proof of \cref{thm:cubes_exist_pronil} is carried out in Sections~\ref{sec:cont_erg_decomp} and \ref{sec:proof_cubes_exist_pronil}. 
The first step is to show the existence -- using crucially the continuous factor maps to pronilfactors in the hypothesis of \cref{thm:cubes_exist_pronil} --  of an ergodic decomposition $x\mapsto \lambda_x^\k$ of the measure $\mu^\k$ that is continuous, and this is carried out in \cref{thm:muk_cont_erg_decomp}.
This decomposition is in turn used in \cref{sec:sigmaka} to construct a disintegration of $\mu^\k$ over the projection to the first coordinate.
In \cref{sec:cubes_exist}, these components are assembled to complete the proof of \cref{thm:cubes_exist_pronil}.

The use of cube systems to find infinite patterns in sets of positive upper Banach density opens new avenues of research, and we discuss some open questions and potential extensions of our work in \cref{sec:questions}. 

For reference, we include a summary of the chain of implications used in the proof of \cref{thm:main_theorem}.
\[
\begin{tikzcd}[column sep=large]
\text{\cref{thm:sigmas_live_on_cubes}} \arrow[Rightarrow, r, "\text{Sec. 7.1}"]
&
\text{\cref{thm:cubes_exist_pronil}} \arrow[Rightarrow, r, "\text{Sec. 5.2}"]
&
\text{\cref{thm:cubes_exist}} \arrow[Rightarrow, r, "\text{Sec. 4.3}"]
&
\text{\cref{thm:main_theorem}}
\end{tikzcd}
\]
\cref{thm:sigmas_live_on_cubes} is proved in \cref{sec:proof_cubes_exist_pronil}.

We thank Dimitrios Charamaras, Andreas Mountakis, Song Shao, Terence Tao and Xiangdong Ye for helpful comments on a preliminary version of this paper, and we are grateful to the anonymous referees for their detailed reports.

\section{Preliminaries}
\label{sec:preliminaries}

In this section we collect preliminary definitions and results from ergodic theory that we use throughout the paper.

\paragraph{Topological systems.}

By a \define{topological system}, we mean a pair $(X,T)$ where $X$ is a compact metric space and $T \colon X \to X$ is a homeomorphism.
We write  $\orbit(x,T) = \{T^n x : n \in \Z \}$ for the \define{orbit} of $x\in X$ under $T$ and
\[
\omg(x,T) = \bigcap_{n\in\N}\overline{ \{ T^j(x) : j \ge n \} }
\]
for the set of forward limit points of $x$ under $T$.
In other words,  $y\in\omega(x,T)$ means that  there is a strictly increasing sequence $c\colon\N\to\N$ with $T^{c(n)}(x)\to y$ as $n\to\infty$.  Note that $x$ need not belong to $\omg(x,T)$.  
The topological system $(X,T)$ is \define{transitive} when there is a point $a \in X$, called a \define{transitive point}, whose orbit $\orbit(a,T)$ is dense in $X$.
A topological system $(X,T)$ is \define{minimal} if every point is transitive, and the system is \define{distal} if
\[
\inf \{ d(T^n x, T^n y) : n \in \N \} > 0
\]
for all $x \ne y \in X$, where $d$ is the metric on $X$.

\paragraph{Measure preserving systems.} 
By a \define{system}, we mean a triple $(X, \mu, T)$ where $(X,T)$ is a topological system and $\mu$ is a $T$-invariant Borel probability measure on $X$.
This means $\mu(T^{-1} D) = \mu(D)$ for all Borel sets $D \subset X$.
We stress that all our measure preserving systems have an underlying topological structure.
Whenever $f\colon  X \to Y$ is measurable and $\mu$ is a Borel measure on $X$ the \define{push forward} of $\mu$ under $f$ is the measure $f \mu$ on $Y$ defined by $(f \mu)(D) = \mu(f^{-1} D)$.

\paragraph{Factor maps.}

We make use of two different types of factor maps between systems.

\begin{Definition}[Measurable factor map]
\label{def:meas_factor_k2}
For a system $(X,\mu,T)$ we say that the system $(Y,\nu,S)$ is a \define{measurable factor} of $(X,\mu,T)$ if there is a measurable map $\pi \colon  X \to Y$, the \define{measurable factor map}, such that $\pi\mu = \nu$ and $(S \circ \pi)(x) = (\pi \circ T)(x)$ for $\mu$-almost every $x \in X$.
\end{Definition}

\begin{Definition}[Continuous factor map]
\label{def:cont_factor_k2}
A system $(Y,\mu,S)$ is a continuous factor of a system $(X,\mu,T)$ if there is a continuous and surjective map $\pi \colon  X \to Y$, the \define{continuous factor map}, with $\pi\mu = \nu$ and $(S \circ \pi)(x) = (\pi \circ T)(x)$ for all $x \in X$.
\end{Definition}

We caution the reader that the distinction between 
Definitions~\ref{def:meas_factor_k2} and~\ref{def:cont_factor_k2} is important throughout the sequel,
and reassure the reader that we  always 
specify measurable or continuous, as appropriate.

\paragraph{Generic points and support of a measure.}
We write $C(X)$ for the space of complex valued continuous functions on the compact metric space $X$, and write $\meas(X)$ for the space of Borel probability measures on $X$ equipped with the weak* topology.
The \define{support} of a Borel probability measure $\mu$ on a compact metric space $X$ is the smallest closed, full measure subset $\supp(\mu)$ of $X$.

Recall that a \define{\Folner{} sequence} is a sequence $\Phi=(\Phi_N)_{N\in\N}$ of finite subsets of $\N$ satisfying \eqref{eqn:folner}.
We stress that all \Folner{} sequences used in this paper are in $\N$, and not in $\Z$, despite the fact that all our systems are invertible.

\begin{Definition}[Generic points]
\label{def:generic_points}
For a system $(X,\mu, T)$ and a \Folner{} sequence $\Phi$, the point $x\in X$ is \emph{generic for $\mu$ with respect to $\Phi$}, written $x\in \gen(\mu,\Phi)$, if
\[
\lim_{N \to \infty} \frac{1}{|\Phi_N|} \sum_{n \in \Phi_N} f(T^nx) = \int_X f \intd\mu
\]
for every $f\in \cont(X)$.
We write $x \in \gen(\mu)$ and say $x$ is \define{generic for $\mu$} if
\[
\lim_{N \to \infty} \frac{1}{N} \sum_{n =1}^Nf(T^nx) = \int_X f \intd\mu
\]
for all $f \in \cont(X)$.
\end{Definition}

\begin{Lemma}
\label{lem:gen_approximates_support}
Fix a system  $(X,\mu, T)$ and a \Folner{} sequence $\Phi$.
If $x \in \gen(\mu,\Phi)$ and $y \in \supp(\mu)$, then $y \in \omg(x,T)$.
\end{Lemma}
\begin{proof}
Fix a compatible metric on $X$ and write $\ball(y,r)$ for the open ball centered at $y$ of radius $r$. For every $\varepsilon > 0$, there exists a non-negative $f \in \cont(X)$ with $f = 1$ on the ball $\ball(y,\varepsilon/2)$ and $f = 0$ outside $\ball(y,\varepsilon)$.  
Now $y \in \supp(\mu)$ implies
\[
\int_X f \intd \mu > 0
\]
and since $x \in \gen(\mu,\Phi)$ there are infinitely many $n \in \N$ with $T^n x \in \ball(y,\varepsilon)$.
\end{proof}

\paragraph{Conditional expectation.}
Fix a system $(X,\mu,T)$.
Whenever $\mathcal{A}$ is a sub-$\sigma$-algebra of the Borel $\sigma$-algebra on $X$, every $f \in \lp^2(X,\mu)$ has a \define{conditional expection} on $\mathcal{A}$ which we write $\E(f \mid \mathcal{A})$ and define to be the orthogonal projection in $\lp^2(X,\mu)$ of $f$ on the closed subspace of $\mathcal{A}$-measurable functions.
We most often apply this when $\mathcal{A}$ is the $\sigma$-algebra of $T$-invariant sets, which we denote $\inv$.

We also need to condition on a factor.
Let $(Y,\nu,S)$ be a measurable factor of $(X,\mu,T)$ via a measurable factor map $\pi$.
Put $\mathcal{A} = \{ \pi^{-1}(F) : F \subset Y \textup{ Borel} \}$.
Given $f \in \lp^2(X,\mu)$ we write $\E(f \mid Y)$ for the conditional expectation $\E(f \mid \mathcal{A})$.
and with the usual identification we may think of $\E(f \mid Y)$ as an element of $\lp^2(Y,\nu)$.
The following properties of the conditional expectation are standard and are made use of occasionally. 
\begin{itemize}
\item
$\vphantom{\displaystyle\int}\E (Tf \mid Y) = \E(f \mid Y) \circ T$ for every $f \in \lp^2(X,\mu)$.
\item
$\displaystyle \int_X \E(f \mid Y) \cdot \E(g \mid Y) \intd \mu = \int_X \E(f \mid Y) \cdot g \intd \mu$ for all $f,g \in \lp^2(X,\mu)$.
\end{itemize}
We refer to~\cite[Chapter 5]{EW11} for further background.
\paragraph{Disintegrations.}

We make significant use of disintegrations of measures over factors.
Fix a system $(X,\mu,T)$ and a probability space $(\Omega,\nu)$.
Whenever we have a map $\omega \mapsto \mu_\omega$ from $\Omega$ to $\meas(X)$ with the properties
\begin{itemize}
\item
$\vphantom{\displaystyle\int}\omega \mapsto \mu_\omega(F)$ is measurable for every Borel $F \subset X$
\item
$\displaystyle \int_X f \intd \mu = \int_\Omega \int_X f \intd \mu_\omega  \intd \nu(\omega)$ for all measurable and bounded $f\colon X \to \C$
\end{itemize}
we call $\omega \mapsto \mu_\omega$ a \define{disintegration} of $\mu$.
The following result says every measurable factor map gives rise to a disintegration that agrees with the conditional expectation on the factor.

\begin{Theorem}[Disintegrations over factor maps, see for example {\cite[Theorem 5.14]{EW11}}]
\label{thm:disintegration}
If $\pi \colon (X,\mu,T) \to (Y,\nu,S)$ is a measurable factor map, then there exists a disintegration $y \mapsto \mu_y$ of $\mu$ defined on $(Y,\nu)$ such that
\[
\E(f \mid Y)(y) = \int_X f \intd \mu_y
\]
for $\nu$-almost every $y\in Y$, whenever $f\colon X \to \C$ is measurable and bounded.
Furthermore, if $y\mapsto\eta_y$ is another such disintegration then $\eta_y=\mu_y$ for $\nu$-almost every $y\in Y$.
\end{Theorem}

\paragraph{Ergodicity and ergodic decompositions.}

Given a topological system $(X,T)$ we say that a $T$-invariant Borel probability measure $\mu$ on $X$ is \define{ergodic} if every  Borel set $D$ satisfying $T^{-1}(D) = D$ has $\mu(D) \in \{0,1\}$.
A system $(X,\mu,T)$ is \define{ergodic} if $\mu$ is ergodic for the topological system $(X,T)$.

We make use of two versions of the pointwise ergodic theorem.
Assume $(X,\mu,T)$ is a system.
Write $\inv$ for the $\sigma$-algebra of $T$-invariant Borel sets.
For every $f$ in $L^1(X,\mu)$ the convergence
\[
\lim_{N \to \infty} \dfrac{1}{N} \sum_{n=1}^N f(T^n x)
=
\E(f \mid \inv)(x)
\]
holds for $\mu$-almost every $x \in X$.
Moreover (see, for example \cite[Section~1]{lindenstrauss}), 
every \Folner{} sequence has a subsequence $\Phi$ with the following property: for every $f$ in $\lp^1(X,\mu)$ the convergence
\[
\lim_{N \to \infty} \dfrac{1}{|\Phi_N|} \sum_{n \in \Phi_N} f(T^n x) = \E(f \mid \inv)(x)
\]
holds for $\mu$-almost every $x\in X$.
We make use of the following standard corollary of the pointwise ergodic theorem. 

\begin{Corollary}
\label{cor:ae_generic}
Assume $(X,\mu,T)$ is an ergodic system.
Every \Folner{} sequence has a subsequence $\Phi$ such that $\mu(\gen(\mu,\Phi)) = 1$.
\end{Corollary}

There is a sense in which every system can be decomposed into ergodic systems.

\begin{Definition}[Ergodic decomposition]
\label{def_ergdec}
Let $(X,\mu,T)$ be a system. An \emph{ergodic decomposition} of $\mu$ is a disintegration $x\mapsto\mu_x$ of $\mu$ defined on $(X,\mu)$ such that for every measurable and bounded $f\colon X \to \C$,
\[
\int_X f \intd \mu_x = \E(f \mid \inv)(x)
\]
holds for $\mu$-almost every $x\in X$.
\end{Definition}

\begin{Theorem}[Existence of ergodic decompositions, see for example {\cite[Theorem 6.2]{EW11}}]
\label{thm:erg_decomp}
Every system $(X,\mu,T)$ has an ergodic decomposition $x \mapsto \mu_x$, satisfying the following properties.
\begin{itemize}
\item
If $x\mapsto \nu_x$ is another ergodic decomposition of $(X,\mu,T)$, then $\mu_x= \nu_x$ for $\mu$-almost every $x\in X$.
\item
We have
\[
\int_X f \intd \mu = \int_X \left( \int_X f \intd \mu_x \right) \intd \mu(x)
\]
for every bounded and measurable $f\colon X\to\C$.  
\item
For $\mu$-almost every $x\in X$, the measure $\mu_x$ is $T$-invariant and ergodic.
\end{itemize}
\end{Theorem}

The following corollary is an immediate consequence of the pointwise ergodic theorem and \cref{thm:erg_decomp}.

\begin{Corollary}
\label{lem:cont_erg_generics}
Fix an ergodic decomposition $x \mapsto \mu_x$ of a system $(X,\mu,T)$.
Then $\mu$-almost every $x\in X$ is generic for $\mu_x$.
\end{Corollary}

\paragraph{Furstenberg's correspondence principle.}
The proof of our main theorem uses a general method for transferring combinatorial problems into dynamical ones known as \define{Furstenberg's correspondence principle}. 
It first appeared in Furstenberg's ergodic proof of \Szemeredi{}'s theorem~\cite{Furstenberg-1977,Furstenberg-book}, and has since formed the basis for numerous applications of ergodic theory to combinatorics and number theory. For our purposes we need a variant of Furstenberg's original method, which we state and prove (see, for example,~\cite[Proposition 3.1]{BHK05}).
\begin{Theorem}
\label{thm_fcp}
Let $A\subset\N$ and suppose $\Phi$ is a \Folner{} sequence in $\N$ such that the limit
\begin{equation}
\label{eqn:main_hyp_02}
\delta=\lim_{N \to \infty} \dfrac{|A \cap \Phi_N|}{|\Phi_N|}
\end{equation}
exists.
There exist an ergodic system $(X,\mu,T)$, a clopen set $E\subset X$, a \Folner{} sequence $\Psi$ in $\N$, and a point $a\in\gen(\mu,\Psi)$ such that $\mu(E)\geq \delta$ and 
\begin{equation}
\label{eqn_fcp_1}
A=\{n\in\N: T^n a\in E\}.
\end{equation}
\end{Theorem}

\begin{proof}
Let $X=\{0,1\}^\Z$ denote the set of all bi-infinite binary sequences, endowed with the product topology, and write $T \colon  \{0,1\}^\Z \to \{0,1\}^\Z$ for the shift map $(Tx)(n) = x(n+1)$.
We can associate to the subset $A\subset\N$ a point $a \in \{0,1\}^\Z$ that represents $A$ via
\begin{equation*}
a(n) =
\begin{cases}
1 &\text{if}~n\in A,
\\
0 &\text{otherwise.}
\end{cases}
\end{equation*}
The set $E = \{ x \in X : x(0) = 1 \}$ is a clopen subset of $X$ and, by construction, \eqref{eqn_fcp_1} is satisfied.
Let $\mu'$ be any weak* accumulation point of the sequence
\[
N \mapsto \mu_N=\frac{1}{|\Phi_N|}\sum_{n\in\Phi_N} \delta_{T^n a}
\]
of measures, where $\delta_x$ denotes the Dirac measure at $x$. Observe that
\[
\mu_N(E)= \dfrac{|A \cap \Phi_N|}{|\Phi_N|}
\]
and hence $\mu'(E)=\delta$ by \eqref{eqn:main_hyp_02}.
Although $\mu'$ is not necessarily ergodic with respect to $T$, since it is $T$-invariant and by the ergodic decomposition there exists an ergodic $T$-invariant Borel probability measure $\mu$ on $X$ with $\mu(E)\geq \mu'(E)=\delta$. Without loss of generality, we can assume that $\mu$ is supported on the orbit closure of $a$ because $\mu'$ is. By \cite[Proposition 3.9]{Furstenberg-book}, there exists a \Folner{} sequence $\Psi$ in $\N$ such that $a\in\gen(\mu,\Psi)$, finishing the proof.
\end{proof}

\section{A discussion of two summands}
%\section{A discussion of \texorpdfstring{$B_1+B_2\subset A$}{B1 + B2 in A}}
\label{sec_B+C}

In this section we present our proof of \cref{thm:main_theorem} in the case $k=2$.
This section is not logically necessary for the full proof of \cref{thm:main_theorem}, but we hope this separate, unadorned treatment of our argument clearly introduces the main ideas behind our approach as preparation for the proof in the general case.
For convenience, let us formulate the statement that we aim to prove as a separate theorem.

\begin{Theorem}
\label{thm:main_theorem_k2}
If $A \subset \N$ has positive upper Banach density
then there are infinite sets $B_1,B_2 \subset \N$ with $B_1 + B_2 \subset A$.
\end{Theorem}

The first step in our proof of \cref{thm:main_theorem_k2} is to recast the problem in dynamical terms using \cref{thm_fcp}.
Doing so paves the way for the language and tools of ergodic theory on which our argument relies.

\begin{Theorem}[Dynamical Reformulation]
\label{thm_dyn_2}
Let $(X,\mu,T)$ be an ergodic system and let $E \subset X$ be a clopen set with $\mu(E)>0$.
If $a\in \gen(\mu,\Phi)$ for some \Folner{} sequence $\Phi$, then there exist strictly increasing sequences $c_1,c_2\colon\N\to\N$ such that 
\begin{equation}
\label{eq:limit-_dyn}
\lim_{j\to\infty}\lim_{m\to\infty}T^{c_1(j)+c_2(m)}a \in E\qquad\text{and}\qquad
\lim_{m\to\infty}\lim_{j\to\infty}T^{c_1(j)+c_2(m)}a\in E.
\end{equation}
\end{Theorem}

\begin{proof}[Proof that \cref{thm_dyn_2} and \cref{thm:main_theorem_k2} are equivalent]
First we deduce \cref{thm_dyn_2} from \cref{thm:main_theorem_k2}.
Fix an ergodic system $(X,\mu,T)$, a clopen set $E$ with $\mu(E) > 0$, and $a \in \gen(\mu,\Phi)$ for some \Folner{} sequence $\Phi$. 
The assumption that $a$ is generic for $\mu$ along $\Phi$ lets us deduce that the set
\[
\{ n \in \N : T^na \in E \}
\]
has positive upper Banach density. It then follows from \cref{thm:main_theorem_k2} that there are infinite sets $B_1,B_2 \subset \N$ with $B_1 + B_2 \subset \{ n \in \N : T^na \in E \}$.
We draw from these sets strictly increasing sequences $c_1,c_2\colon\N\to\N$ with $T^{c_1(j) + c_2(m)}a \in E$ for all $j,m \in \N$ so that, after passing to suitable subsequences such that all the limits involved exist,  \eqref{eq:limit-_dyn} follows immediately.

To deduce \cref{thm:main_theorem_k2} from \cref{thm_dyn_2}, fix $A\subset\N$ with positive upper Banach density. 
\cref{thm_fcp} yields an ergodic system $(X,\mu,T)$, a point $a \in \gen(\mu,\Phi)$ for some \Folner{} sequence $\Phi$, and a clopen set $E$ with $\mu(E) > 0$ such that
\[
\{ n \in \Z : T^n a \in E \} = A.
\]
It follows from \cref{thm_dyn_2} that there are strictly increasing sequences $c_1,c_2\colon\N\to\N$ such that \eqref{eq:limit-_dyn} holds.

We now apply alternately the limits in \eqref{eq:limit-_dyn} to construct subsequences $b_1,b_2$ of $c_1,c_2$  such that $b_1(j)+b_2(m)\in\{n\in\N: T^n a\in E\}$ for all $j,m\in\N$.
Begin by selecting $b_1(1)$ in $\{c_1(j): j \in \N \}$ such that 
\[
\lim_{m\to\infty} T^{b_1(1)+c_2(m)}a \in E.
\]
Then let $b_2(1)$ from $\{c_2(m) : m\in\N\}$ satisfy $b_1(1)+b_2(1) \in\{n\in\N: T^n a \in E\}$ and 
\[
\lim_{j\to\infty} T^{c_1(j)+b_2(1)} a \in E.
\]
Next, we take $b_1(2) > b_1(1)$ to be any element in $\{c_1(n) : n\in\N\}$ with the property that $b_1(2) + b_2(1) \in\{n\in\N: T^n a\in E\}$ and 
\[
\lim_{m\to\infty} T^{b_1(2)+c_2(m)}a \in E.
\]
Thereafter, we take $b_2(2)$ from $\{c_2(n) : n\in\N\}$ with $b_1(1)+b_2(2), b_1(2)+b_2(2)\in\{n\in\N: T^n a\in E\}$
and
\[
\lim_{j \to \infty} T^{c_1(j)+b_2(2)}a \in E.
\]
Continuing this procedure by induction yields strictly increasing sequences $b_1,b_2\colon\N\to\N$ with $T^{b_1(j)+b_2(m)}a$ belonging to $E$ for all $j,m\in\N$. Since $A=\{ n \in \Z : T^na \in E \} $, we conclude ${b_1(j)+b_2(m)}\in A$ for all $j,m\in\N$, finishing the proof.
\end{proof}

\begin{Remark}
When $a \in \supp(\mu)$ there is a strictly increasing sequence $b\colon\N\to\N$ and some $c \in \N$ with
\[
\{ b(i_1) + \cdots + b(i_n) + c : i_1 < \cdots < i_n \in \N, n \in \N \} \subset A, 
\]
which is to say that $A$ contains a shift of an IP set, a significantly stronger conclusion than that of \cref{thm:main_theorem_k2}.
We stress that there is no reason to expect that $a$ belongs to $\supp(\mu)$, and in fact it is known that not every set with positive upper density contains a shift of an IP set. In such cases, necessarily $a\notin\supp(\mu)$, and this introduces new complications in our proofs. 
\end{Remark}

The remainder of this section is dedicated to proving \cref{thm_dyn_2} using the techniques of ergodic theory.
In fact, we prove a stronger result: under the hypotheses of \cref{thm_dyn_2} there exist strictly increasing sequences $c_1,c_2 \colon \N \to \N$ such that the double limits in \eqref{eq:limit-_dyn} not only belong to $E$ but are also equal to one another, that is,
\begin{equation}
\label{eq:limit-_dyn2}
\lim_{j\to\infty}\lim_{m\to\infty}T^{c_1(j)+c_2(m)}a
=
\lim_{m\to\infty}\lim_{j\to\infty}T^{c_1(j)+c_2(m)}a.
\end{equation}
To better analyze this transposition of limits, it is convenient to keep track of the intermediate limit points.
For this purpose, when \eqref{eq:limit-_dyn2} holds we introduce the notation
\begin{align*}
x_{00}&= a
\\
x_{01}&= \lim_{m\to\infty}T^{c_2(m)}a
\\
x_{10}&=\lim_{j\to\infty}T^{c_1(j)}a
\\
x_{11}&= \lim_{j\to\infty}\lim_{m\to\infty}T^{c_1(j)+c_2(m)}a
=
\lim_{m\to\infty}\lim_{j\to\infty}T^{c_1(j)+c_2(m)}a.
\end{align*}
In terms of this notation \eqref{eq:limit-_dyn2} is equivalent to the following identities:
\begin{equation}
\label{eqn_2_dim_erdos_cube}
\begin{gathered}
\lim_{j\to\infty}(T \times T)^{c_1(j)}(x_{00},x_{01}) = (x_{10},x_{11}) \\
\lim_{m\to\infty}(T \times T)^{c_2(m)}(x_{00},x_{10}) = (x_{01},x_{11})
\end{gathered}
\end{equation}
Following the notation introduced in~\cite{HK-05}, we write 
$\kone = \{0, 1\}$ and $\two = \{00, 01, 10, 11\}$.
Tuples $(x_{00},x_{01},x_{10},x_{11})\in X^\two$ for which \eqref{eqn_2_dim_erdos_cube} holds are central to our approach and we formalize them as follows.

\begin{Definition}[$2$-dimensional \Erdos{} cubes]
\label{def_erdos_cubes_2d_2}
Given a topological system $(X,T)$ we call any point $x=(x_{00},x_{01},x_{10},x_{11}) \in X^\two$ satisfying \eqref{eqn_2_dim_erdos_cube} for some strictly increasing sequences $c_1, c_2\colon\N\to\N$ a \define{$2$-dimensional \Erdos{} cube}.
\end{Definition}

Observe that \cref{def_erdos_cubes_2d_2} matches \cref{def_erdos_cubes_2d}.
Indeed, \eqref{eqn_2_dim_erdos_cube} is equivalent to the assertion that the forward orbit of $(x_{00},x_{01})$ visits every neighborhood of $(x_{10},x_{11})$ infinitely often, and that the forward orbit of $(x_{00},x_{10})$ visits every neighborhood of $(x_{01},x_{11})$ infinitely often. This can also be written as \begin{equation}
\label{eqn_2_dim_erdos_cube_2}
\begin{gathered}
(x_{10},x_{11}) \in 
\omega((x_{00},x_{01}),T \times T)
\quad\text{and}\quad
\bigl(x_{01},x_{11}\bigr) \in  
\omega\bigl((x_{00},x_{10}),T \times T\bigr).
\end{gathered}
\end{equation}
We think of the coordinates $x_{00},x_{01},x_{10},x_{11}$ in a $2$-dimensional \Erdos{} cube as forming the vertices of a square
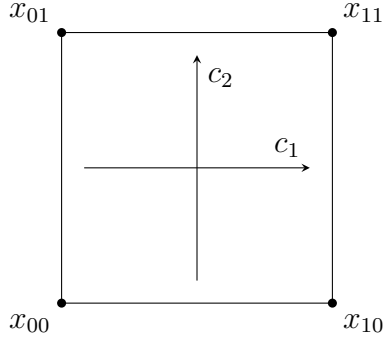
\[
\begin{tikzpicture}[scale=3]
\coordinate (00) at (0,0);
\coordinate (10) at (1.2,0);
\coordinate (01) at (0,1.2);
\coordinate (11) at (1.2,1.2);
\draw (00) -- (10) -- (11) -- (01) -- (00);
\fill (00) circle (0.02);
\fill (10) circle (0.02);
\fill (01) circle (0.02);
\fill (11) circle (0.02);
\node[anchor=north east] at (00) {$x_{00}$};
\node[anchor=north west] at (10) {$x_{10}$};
\node[anchor=south east] at (01) {$x_{01}$};
\node[anchor=south west] at (11) {$x_{11}$};
\draw[-stealth] (0.1,0.6) to (1.1,0.6);
\node[anchor=south] at (1.0,0.6) {$c_1$};
\draw[-stealth] (0.6,0.1) to (0.6,1.1);
\node[anchor=west] at (0.6,1.0) {$c_2$};
\end{tikzpicture}
\]
% \[
% \begin{tikzpicture}[scale=3]
% \coordinate (00) at (0,0);
% \coordinate (10) at (1.5,0);
% \coordinate (01) at (0,1);
% \coordinate (11) at (1.5,1);
% \draw (00) -- (10) -- (11) -- (01) -- (00);
% \fill (00) circle (0.02);
% \fill (10) circle (0.02);
% \fill (01) circle (0.02);
% \fill (11) circle (0.02);
% \node[anchor=north east] at (00) {$x_{00}$};
% \node[anchor=north west] at (10) {$x_{10}$};
% \node[anchor=south east] at (01) {$x_{01}$};
% \node[anchor=south west] at (11) {$x_{11}$};
% \draw[-stealth] (0.1,0.5) to (1.4,0.5);
% \node[anchor=south] at (1.3,0.5) {$c_1$};
% \draw[-stealth] (0.75,0.1) to (0.75,0.9);
% \node[anchor=west] at (0.75,0.8) {$c_2$};
% \end{tikzpicture}
% \]
and the sequences $c_1,c_2\colon\N\to\N$ as describing iterations of $T \times T$ at which the left and bottom sides approximate the right and top sides respectively.

We can now state the main dynamical theorem of this section, which is a strengthening of \cref{thm_dyn_2}.

\begin{Theorem}[Existence of $2$-dimensional \Erdos{} cubes]
\label{thm:2-dim_cubes_exist}
Assume $(X,\mu,T)$ is an ergodic system, $a \in \gen(\mu,\Phi)$ for some \Folner{} sequence $\Phi$, and $E\subset X$ is an open set with $\mu(E)>0$.
There exists a $2$-dimensional \Erdos{} cube $(x_{00},x_{01},x_{10},x_{11})\in X^\two$ with $x_{00}=a$ and $x_{11}\in E$.
\end{Theorem}

Although it can already be inferred from our discussion, we include a quick and self-contained proof of the fact that \cref{thm:2-dim_cubes_exist} implies \cref{thm_dyn_2}.

\begin{proof}[Proof that \cref{thm:2-dim_cubes_exist} implies \cref{thm_dyn_2}]
If $(x_{00},x_{01},x_{10},x_{11})$ is a $2$-dimensional \Erdos{} cube then, by definition, there exist increasing sequences $c_1,c_2\colon\N\to\N$ such that \[
\lim_{n\to\infty}T^{c_1(n)}x_{00}=x_{10}, \quad \textup{and} \quad \lim_{m\to\infty}T^{c_2(m)}x_{10}=x_{11}.
\]
Since $x_{00}=a$ and $x_{11}\in E$, we conclude that \[
\lim_{m\to\infty}\lim_{n\to\infty}T^{c_1(n)+c_2(m)}a \in E.
\] 
Similarly, we have
\[
\lim_{m\to\infty}T^{c_2(m)}x_{00}=x_{01}  \quad\textup{and}\quad \lim_{n\to\infty}T^{c_1(n)}x_{01}=x_{11},
\]
which implies
\[
\lim_{n\to\infty}\lim_{m\to\infty}T^{c_1(n)+c_2(m)}a\in E.
\]
This completes the proof of \eqref{eq:limit-_dyn}.
\end{proof}

Before embarking on the proof of \cref{thm:2-dim_cubes_exist}, which involves finding an \Erdos{} cube with prescribed first and last coordinates, we explain why there is a natural abundance of \Erdos{} cubes when these extra restrictions are omitted.
A key role is played by the \define{cubic measures}, which are special measures introduced in~\cite{HK-05} to analyze multiple ergodic averages.
We recall the definition of one- and two-dimensional cubic measures 
(the definition for higher dimensions is given in \cref{sec:cubic-measures}).

\begin{Definition}[1-dimensional cubic measure]
Let $(X,\mu,T)$ be an ergodic system.
The measure $\mu^\kone=\mu\times\mu$ on $X^\kone$
is the \define{1-dimensional cubic measure} of $(X,\mu,T)$. It gives rise to the \define{1-dimensional cube system} $(X^\kone, \mu^\kone,T^\kone)$, where $T^\kone=T\times T$.
\end{Definition}

Fix an ergodic decomposition $y \mapsto (\mu^\kone)_y$ of $\mu^\kone$
 using \cref{thm:erg_decomp}.
As stipulated in Section~\ref{sec:preliminaries}, we view ergodic decompositions as disintegrations over the invariant factor, which means $y \mapsto (\mu^\kone)_y$ is an $\mu^\kone$-almost everywhere defined map on $X^\kone$ satisfying 
\begin{equation}
\label{eqn_1-dim_cubic_measure_erg_decomp}
\mu^\kone = \int_{X^\kone} (\mu^\kone)_y \intd \mu^\kone(y),
\end{equation}
and there is a full $\mu^\kone$-measure set of $y\in X^\kone$ for which $(\mu^\kone)_y$ is $T^\kone$-invariant and ergodic and $y$ is generic for $(\mu^\kone)_y$.

\begin{Definition}[2-dimensional cubic measure]
Given an ergodic system $(X,\mu,T)$ and an ergodic decomposition $y \mapsto (\mu^\kone)_y$ of $\mu^\kone$, the \define{2-dimensional cubic measure} of $(X,\mu,T)$, denoted by $\mu^\two$, is the measure on $X^\two$ defined by
\begin{equation}
\label{eqn:2d_cube_measure_k2}
\mu^\two = \int_{X^\kone} (\mu^\kone)_y \times (\mu^\kone)_y \intd \mu^\kone(y).
\end{equation}
Writing $T^\two=T\times T\times T\times T$, the system $(X^\two,\mu^\two,T^\two)$ is the \define{2-dimensional cube system} associated to $(X,\mu,T)$.
\end{Definition}

In general, the measure $\mu^\two$ is not equal to the product measure $\mu^\kone \times \mu^\kone$, 
but it is still invariant under the diagonal transformation $T^\two = T \times T \times T \times T$.
Particularly important to our discussion is the fact that $\mu^\two$ possesses special symmetries.
Consider the permutation on the set $\two$ given by $00\mapsto 00$, $01\mapsto 10$, $10\mapsto 01$, and $11\mapsto 11$.
This permutation naturally induces a map $\phi\colon X^\two\to X^\two$ defined by
\begin{equation}
\label{eqn_2-dim_permutation}
\phi(x_{00}, x_{01}, x_{10}, x_{11})=(x_{00}, x_{10}, x_{01}, x_{11}).
\end{equation}
One can show that $\phi$ is always an automorphism of the 2-dimensional cube system $(X^\two,\mu^\two,T^\two)$, which means that $\phi$ commutes with $T^\two$ and preserves the measure $\mu^\two$, i.e., 
\begin{equation}
\label{eqn_2-dim-cubic-measure-permutation-symmetry}
\phi(\mu^\two)=\mu^\two.
\end{equation}
A proof of this \define{permutation symmetry} of $\mu^\two$ is given in~\cite[Proposition 8.8]{HK-book}.

Next, we describe the connection between cubic measures and \Erdos{} cubes. 
We begin with the combinatorially uninteresting case of 1-dimensional \Erdos{} cubes.
These are the points $x \in X^\kone$ for which there is a strictly increasing sequence $c_1\colon\N\to\N$ with
\[
\lim_{n \to \infty} T^{c_1(n)}(x_0) = x_1.
\]
In view of \cref{lem:gen_approximates_support}, if $x_0\in \gen(\mu,\Phi)$ and $x_1 \in \supp(\mu)$ then the pair $(x_0,x_1)$ is a 1-dimensional \Erdos{} cube.
After passing to a subsequence of $\Phi$ if necessary, the set $\gen(\mu,\Phi)\times \supp(\mu)$ has full measure with respect to $\mu^\kone$ by \cref{cor:ae_generic}.
This proves that $\mu^\kone$-almost every point in $X^\kone$ is a 1-dimensional \Erdos{} cube.

The extra symmetry provided by \eqref{eqn_2-dim-cubic-measure-permutation-symmetry} allows us to supplement the previous argument to produce 2-dimensional \Erdos{} cubes.

\begin{Proposition}
\label{prop_a.e.-point_is_Erdos_cube_2-dim}
Let $(X,\mu,T)$ be an ergodic system.
Then $\mu^\two$-almost every $x\in X^\two$ is a $2$-dimensional \Erdos{} cube.
\end{Proposition}

\begin{proof}
Fix an ergodic decomposition $y \mapsto (\mu^\kone)_y$ of $\mu^\kone$.
Since $(\mu^\kone)_y$ is ergodic for $\mu^\kone$-almost every $y \in X^\kone$ we can repeat the previous argument for finding $1$-dimensional \Erdos{}~cubes to show that $(\mu^\kone)_y \times (\mu^\kone)_y$ gives full measure to
\begin{equation}
\label{eqn:half_of_cube_k2}
\left\{ (x_{00}, x_{01}, x_{10}, x_{11}) \in X^\two : (x_{10},x_{11}) \in \omega((x_{00},x_{01}),T \times T) \right\}
\end{equation}
for $\mu^\kone$-almost every $y\in X^\kone$.
It is then immediate from \eqref{eqn:2d_cube_measure_k2} that $\mu^\two$ gives full measure to \eqref{eqn:half_of_cube_k2}.
We now use \eqref{eqn_2-dim-cubic-measure-permutation-symmetry}.
It follows that \eqref{eqn:half_of_cube_k2} and its inverse image under $\phi$ both have full measure with respect to $\mu^\two$.
Since $x \in X^\two$ is a 2-dimensional \Erdos{} cube if and only if $x$ and $\phi(x)$ belong to \eqref{eqn:half_of_cube_k2}, we conclude that $\mu^\two$-almost every point is a 2-dimensional \Erdos{} cube.
\end{proof}

We now embark on the discussion of the more delicate issue of producing \Erdos{} cubes with prescribed first coordinate.
Once this is in hand, the proof of \cref{thm:2-dim_cubes_exist} is straightforward.

In view of \cref{prop_a.e.-point_is_Erdos_cube_2-dim}, we can attempt to obtain \Erdos{} cubes with a prescribed first coordinate by disintegrating $\mu^\two$ with respect to the projection $(x_{00},x_{01},x_{10},x_{11})\mapsto x_{00}$ from $X^\two$ to its first coordinate. 
Since the coordinate projection $x \mapsto x_{00}$ is a measurable factor map from $(X^\two,\mu^\two,T^\two)$ to $(X,\mu,T)$, we can apply \cref{thm:disintegration} to obtain a disintegration of $\mu^\two$ over this map: a collection of measures $\sigma^\two_t$ on $X^\two$ defined for $\mu$-almost every $t \in X$. 
Any such disintegration satisfies
\[
\mu^\two = \int_X \sigma^\two_t \intd \mu(t),
\]
which, combined with \cref{prop_a.e.-point_is_Erdos_cube_2-dim}, shows $\sigma^\two_t$-almost every point is a 2-dimensional \Erdos{} cube for $\mu$-almost every $t\in X$.
In fact
\[
\sigma^\two_t( \{ x \in X^\two : x_{00} = t \} ) = 1
\]
for $\mu$-almost every $t$, and we conclude that for $\mu$-almost every $t \in X$ there are 2-dimensional \Erdos{} cubes $x \in X^\two$ with $x_{00} = t$.
It can also be shown that the push forward of $\sigma^\two_t$ to the \textit{last} coordinate via the map $x \mapsto x_{11}$ is equal to $\mu$ for $\mu$-almost every $t$, and therefore one can find for $\mu$-almost every $t \in X$ an \Erdos{} cube $x$ with $x_{00} = t$ and $x_{11} \in E$.

Unfortunately, this argument is not enough to prove \cref{thm:2-dim_cubes_exist} because \cref{thm:disintegration} only defines $\sigma^\two_t$ for $\mu$-almost every $t\in X$: the measure $\sigma^\two_a$ might not be defined.
Thus, to proceed this way we must first define $\sigma^\two_t$ for \textit{every} point $t \in X$.
We are able to do this when the system $(X,\mu,T)$ admits a continuous factor map to a certain measurable factor of the ergodic system $(X,\mu,T)$ called its Kronecker factor.
Informally, it is the largest factor of $(X,\mu,T)$ that is isomorphic to a rotation on a compact abelian group.

\begin{Definition}[Group rotation]
A system $(Z,m,R)$ is a \define{group rotation} if $Z$ is a compact abelian group, $m$ is its Haar measure, and the transformation $R$ has the form $R(z) = z + \alpha$ for some fixed $\alpha \in Z$.
\end{Definition}

\begin{Theorem}
\label{thm:kronecker_exists_k2}
Fix an ergodic system $(X,\mu,T)$.
There is an ergodic group rotation $(Z,m,R)$ that is a measurable factor of $(X,\mu,T)$ with the property that, for every $f,g$ in $L^\infty(X,\mu)$ we have
\[
\lim_{N \to \infty} \dfrac{1}{N} \sum_{n =1}^N f(T^n x_0) \cdot g(T^n x_1) 
=
\lim_{N \to \infty} \dfrac{1}{N} \sum_{n =1}^N T^n \E(f \mid Z)(x_0) \cdot T^n \E(g \mid Z)(x_1)
\]
for $\mu\times\mu$-almost every $(x_0,x_1) \in X \times X$.
\end{Theorem}
\begin{proof}
This is a consequence of the pointwise ergodic theorem applied to $T \times T$ together with the description~\cite[Lemma~4.21]{Furstenberg-book} of the projection onto the invariant factor in a product system.
\end{proof}

We refer to the group rotation factor $(Z,m,R)$ of $(X,\mu,T)$ in \cref{thm:kronecker_exists_k2} as the \define{Kronecker factor} of $(X,\mu,T)$.
\cref{thm:kronecker_exists_k2} gives only a measurable factor map to the Kronecker factor, but to define $\sigma^\two_t$ for all $t \in X$ it turns out that we need this factor map to be continuous.
Fortunately, we can always pass to an extension of our ergodic system having that property.
We prove this at the end of the section, continuing for the moment under the assumption that $(X,\mu,T)$ has its Kronecker factor $(Z,m,R)$ as a continuous factor via a continuous factor map $\pi$.
Under this assumption we are able to improve upon \cref{thm:erg_decomp}: we give a disintegration of $\mu^\kone$ that is defined on all of $X^\kone$, continuous, and an ergodic decomposition of $\mu^\kone$.  
This result is the main technical step in the proof of \cref{thm:2-dim_cubes_exist}.

\begin{Proposition}
\label{prop_cont_kronecker_implies_cont_erg_decomp}
Let $(X,\mu,T)$ be an ergodic system and assume there is a continuous factor map $\pi$ to its Kronecker factor $(Z,m,R)$.
There exists a continuous map $x\mapsto \lambda_{x}^\kone$ from $X^\kone$ to $\meas(X^\kone)$ with the following properties.
\begin{enumerate}
[label=(\roman{enumi}),ref=(\roman{enumi}),leftmargin=*]
\item
\label{itm_ced_d2_erg_decomp}
The map $x\mapsto \lambda_{x}^\kone$ is an ergodic decomposition of $\mu^\kone$ as in \cref{def_ergdec}.
\item\label{itm_ced_d2_kronecker_saturated}
We have that 
$\lambda^\kone_x = \lambda^\kone_y$ whenever $\bigl(\pi(x_0), \pi(x_1)\bigr) = \bigl(\pi(y_0), \pi(y_1)\bigr)$.
\item\label{itm_ced_d2_pushforward}
For all $x \in X^\kone$, we have that $T^\kone\lambda_x^\kone=\lambda^\kone_{T^\kone(x)}=\lambda_x^\kone$.
\item\label{itm_ced_d2_marginal_proj}
For all $x \in X$ and all Borel $F \subset X$, we have that 
$\lambda^\kone_x(X \times F) = \mu(F)$.
\end{enumerate}
\end{Proposition}

\begin{Remark}
\label{rem_ced}
Note that any ergodic decomposition of $\mu^\kone$ satisfies properties \ref{itm_ced_d2_kronecker_saturated} through \ref{itm_ced_d2_marginal_proj} for almost all $x,y\in X^\kone$; the fact that $x\mapsto \lambda_{x}^\kone$ satisfies these properties for all $x,y\in X^\kone$ is what makes it special.
We refer to $x\mapsto \lambda_x^\kone$ as a \define{continuous ergodic decomposition}, and discuss the general case in \cref{sec:cont_erg_decomp}.
\end{Remark}

\begin{proof}[Proof of \cref{prop_cont_kronecker_implies_cont_erg_decomp}]
Apply \cref{thm:disintegration} to get a disintegration $z \mapsto \eta_z$ of $\mu$ over the continuous factor map $\pi$ from $(X,\mu,T)$ to its Kronecker factor $(Z,m,R)$.
Define
\begin{equation}
\label{eqn:lambda_2_dim_definition_for_section_2_is_this_unqiue_yet}
\lambda^\kone_x = \int_Z \eta_{z + \pi(x_0)} \times \eta_{z+\pi(x_1)} \intd m(z)
\end{equation}
for every $x \in X^\kone$.

We first note that for each $x \in X^\kone$ the measures $\eta_{z + \pi(x_0)}$ and $\eta_{z + \pi(x_1)}$ are defined for $m$-almost every $z \in Z$ and therefore \eqref{eqn:lambda_2_dim_definition_for_section_2_is_this_unqiue_yet} is well-defined.
To prove that $x \mapsto \lambda^\kone_x$ is continuous first note that uniform continuity implies
\[
(v,w) \mapsto \int_Z f(z + w) \cdot g(z+v) \intd m(z)
\]
from $Z$ to $\C$ is continuous whenever $f,g \colon  Z \to \C$ are continuous.
An approximation argument then gives continuity for every $f,g \in \lp^2(Z,m)$. 
In particular,
\[
x \mapsto \int_Z \E(f \mid Z)(z + \pi(x_0)) \cdot \E(g \mid Z)(z + \pi(x_1)) \intd m(z)
\]
from $X^\kone$ to $\C$ is continuous whenever $f,g \colon  X \to \C$ are continuous, which in turn implies continuity of \eqref{eqn:lambda_2_dim_definition_for_section_2_is_this_unqiue_yet}.

To prove that $x\mapsto \lambda_{x}^\kone$ is an ergodic decomposition we calculate
\begin{align*}
&
\int_{X^\kone} \int_Z \eta_{z + \pi(x_0)} \times \eta_{z + \pi(x_1)} \intd m(z) \intd \mu^\kone(x)
\\
=
&
\int_Z \int_X \eta_{z + \pi(x_0)} \intd \mu(x_0) \times \int_X \eta_{z + \pi(x_1)} \intd \mu(x_1) \intd m(z), 
\end{align*}
which is equal to $\mu^\kone$ because both inner integrals are equal to $\mu$.
We conclude that
\begin{equation}
\label{eq_continuousergodicdecompositionofmu1}
\mu^\kone = \int_{X^\kone}\lambda^\kone_x \intd \mu^\kone(x), 
\end{equation}
which shows $x \mapsto \lambda^\kone_x$ is a disintegration of $\mu^\kone$.

We are left with verifying that
\[
\int_{X^\kone} F \intd \lambda^\kone_x = \E(F \mid \inv)(x)
\]
for $\mu^\kone$-almost every $x\in X^\kone$ whenever $F \colon  X^\kone \to \C$ is measurable and bounded.
Recall that $\inv$ denotes the $\sigma$-algebra of $T^\kone$-invariant sets.
Fix such an $F$.
It follows from the pointwise ergodic theorem that
\[
\lim_{N \to \infty} \dfrac{1}{N} \sum_{n=1}^N F( T^n x_0, T^n x_1)
=
\E(F \mid \inv)(x)
\]
for $\mu^\kone$-almost every $x=(x_0,x_1)\in X^\kone$.
We therefore wish to prove that
\[
\int_{X^\kone} F \intd \lambda^\kone_x = \lim_{N \to \infty} \dfrac{1}{N} \sum_{n=1}^N F( T^n x_0, T^n x_1)
\]
holds for $\mu^\kone$-almost every $x \in X^\kone$.

By an approximation argument it suffices to verify that 
\begin{equation}
\label{eqn:proving_ergodic_k2}
\int_{X^\kone} f \otimes g \intd \lambda^\kone_x = \lim_{N \to \infty} \dfrac{1}{N} \sum_{n=1}^N f( T^n x_0) \cdot g(T^n x_1)
\end{equation}
holds for $\mu^\kone$-almost every $x \in X^\kone$ whenever $f,g$ belongs to $\lp^\infty(X,\mu)$.

In fact, in view of \cref{thm:kronecker_exists_k2}, only the conditional expectations of $f$ and $g$ on the Kronecker factor contribute to the right-hand side.
We therefore have for any $f,g \in \lp^2(X,\mu)$ that
\[
\lim_{N \to \infty} \dfrac{1}{N} \sum_{n=1}^N f(T^n x_0) \cdot g(T^n x_1)
=
\lim_{N \to \infty} \dfrac{1}{N} \sum_{n=1}^N \E(f \mid Z)(T^n x_0) \cdot \E(g \mid Z)(T^n x_1)
\]
for $\mu^\kone$-almost every $x\in X^\kone$.
Now
\[
\lim_{N \to \infty} \dfrac{1}{N} \sum_{n=1}^N \phi(T^n z_0) \cdot \psi(T^n z_1)
=
\int_Z \phi(s + z_0) \cdot \psi(s + z_1) \intd m(s)
\]
for any $\phi,\psi$ in $\lp^2(Z,m)$
using properties of rotations on a compact abelian group.
Taking $\phi = \E(f \mid Z)$ and $\psi = \E(g \mid Z)$ gives
\[
\lim_{N \to \infty} \dfrac{1}{N} \sum_{n=1}^N \E(f \mid Z)(T^nx_0) \cdot \E(g \mid Z)(T^nx_1)
=
\int_{X^\kone} f \otimes g \intd \lambda^\kone_x
\]
for $\mu^\kone$-almost all $x\in X^\kone$.
As a consequence we have for every $f,g \in \lp^2(X,\mu)$ that
\[
\int_{X^\kone} f \otimes g \intd \lambda^\kone_x = \E( f \otimes g \mid \inv)(x)
\]
for $\mu^\kone$-almost every $x\in X^\kone$, proving property~\ref{itm_ced_d2_erg_decomp}.

It is immediate from \eqref{eqn:lambda_2_dim_definition_for_section_2_is_this_unqiue_yet} that $\pi(x_0) = \pi(y_0)$ and $\pi(x_1) = \pi(y_1)$ together imply the measures $\lambda^\kone_x$ and $\lambda^\kone_y$ are equal, verifying property~\ref{itm_ced_d2_kronecker_saturated}.

Property~\ref{itm_ced_d2_pushforward} follows from \cref{thm:erg_decomp} and the continuity of $x\mapsto \lambda_{x}^\kone$. Finally, property ~\ref{itm_ced_d2_pushforward} follows from
\[
\lambda^\kone_x(X \times F)
=
\int_Z \eta_{z + \pi(x_0)}(X) \cdot \eta_{z + \pi(x_1)}(F) \intd m(z)
=
\int_Z \eta_z(F) \intd m(z) = \mu(F)
\]
because $z \mapsto \eta_z$ is a disintegration of $\mu$.
\end{proof}

We next use \cref{prop_cont_kronecker_implies_cont_erg_decomp} to define the measures $\sigma^\kone_t$ directly for \textit{every} point $t \in X$, dispensing with our direct application of \cref{thm:disintegration}.

\begin{Definition}
\label{def:sigmaka_k2}
Fix an ergodic system $(X,\mu,T)$ with a continuous factor map to its Kronecker factor and let $\lambda^\kone\colon X^\kone\to\meas(X^\kone)$ be the disintegration given by \cref{prop_cont_kronecker_implies_cont_erg_decomp}.
We define a measure $\sigma^\two_t$ on $X^\two$ by
\begin{equation}
\label{eqn:sigmaka_def_k2}
\sigma^\two_t = \int_X \delta_{(t,x)} \times \lambda^\kone_{(t,x)} \intd \mu(x)
\end{equation}
for all $t \in X$.
\end{Definition}

For the rest of the section all references to measures $\sigma^\two_t$ are made with respect to \cref{def:sigmaka_k2}.
The following properties of the measures $\sigma^\two_t$ are essential for producing \Erdos{} cubes with prescribed coordinates.
Writing $F^* \colon  X^\two \to X^{\two \setminus \{00\}}$ for the projection
\begin{equation}
\label{eqn:ignore_first_k2}
F^* (x_{00},x_{01},x_{10},x_{11})
=
(x_{01},x_{10},x_{11})
\end{equation}
 note in particular the last property, which allows us to make certain statements about $\sigma^\two_t$-almost every point even when $t$ does not lie in the support of $\mu$.

\begin{Theorem}
\label{prop_2-dim_sigma_rep}
Let $(X,\mu,T)$ be an ergodic system and assume there is a continuous factor map $\pi$ to its Kronecker factor.
Then the map $t \mapsto \sigma^\two_t$ is continuous, gives full measure to $\{ x \in X^\two : x_{00} = t \}$, satisfies
\begin{equation}
\label{eqn:sigma_props_k2}
\int_X \sigma^\two_t \intd \mu(t) = \mu^\two
\end{equation}
and the following properties.
\begin{enumerate}
[label=(\roman{enumi}),ref=(\roman{enumi}),leftmargin=*]
\item\label{itm_cd1c_i}
The push forward $T^\two \sigma^\two_t$ equals $\sigma^\two_{T(t)}$ for every $t \in X$.
\item\label{itm_cd1c_ii}
The push forwards $F^* \sigma^\two_t$ and $F^* \sigma^\two_s$ are equal whenever $\pi(t) = \pi(s)$.
\end{enumerate}
\end{Theorem}

\begin{Remark}
\label{rem_cd1c}
In analogy to \cref{rem_ced}, let us point out that any decomposition of $\mu^\two$ with respect to the first coordinate satisfies properties \ref{itm_cd1c_i} and \ref{itm_cd1c_ii} for almost all $t\in X$. The fact that $t\mapsto \sigma^\two_t$ satisfies these properties for all $t\in X$ is what distinguishes it from a generic disintegration.
\end{Remark}

\begin{proof}[Proof of \cref{prop_2-dim_sigma_rep}]
The fact that the set of points in $X^\two$ whose first coordinate equals $t$ has full measure with respect to $\sigma^\two_t$ is immediate from \eqref{eqn:sigmaka_def_k2}.  
To prove \eqref{eqn:sigma_props_k2} we calculate
\begin{align*}
\int_X \sigma^\two_t \intd \mu(t)
&= \int_X\int_X \delta_{(t,s)} \times \lambda^\kone_{(t,s)}\intd\mu(t)\d\mu(s)
\\
&= \int_{X^\kone} \delta_{(t,s)} \times \lambda^\kone_{(t,s)}\d\mu^\kone(t,s)
\\
&= \int_{X^\kone} \biggl(\int_{X^\kone} \delta_{(u,v)} \times \lambda^\kone_{(u,v)} \d\lambda^\kone_{(t,s)}(u,v)\biggr)\d\mu^\kone(t,s)
\\
&= \int_{X^\kone} \biggl(\int_{X^\kone} \delta_{(u,v)} \times \lambda^\kone_{(t,s)} \d\lambda^\kone_{(t,s)}(u,v)\biggr)\d\mu^\kone(t,s)
\\
&= \int_{X^\kone} \lambda^\kone_{(t,s)} \times \lambda^\kone_{(t,s)}\d\mu^\kone(t,s)\,=\mu^\two
\end{align*}
where the fourth  equality follows from
\[
\lambda^\kone_{(t,s)} ( \{ (u,v) : \lambda^\kone_{(t,s)} = \lambda^\kone_{(u,v)} \} ) = 1
\]
for $\mu^\kone$-almost every $(t,s)$, which in turn is a consequence of part \ref{itm_ced_d2_pushforward} of \cref{prop_cont_kronecker_implies_cont_erg_decomp} and ergodicity of $\lambda^\kone_{(t,s)}$ for $\mu^\kone$-almost every $(t,s)$.
\end{proof}

We have reduced the proof of \cref{thm:2-dim_cubes_exist} to the statement that $\sigma^\two_a$-almost every point is an \Erdos{} cube.
The proof of \cref{prop_a.e.-point_is_Erdos_cube_2-dim} made essential usage of the fact that $\mu^\two$ was symmetric with respect to the permutation map \eqref{eqn_2-dim_permutation}.
Our penultimate ingredient in the proof of \cref{thm:2-dim_cubes_exist} is that the measures $\sigma^\two_t$ have the same symmetries.

To prove that \textit{all} of the measures $\sigma^\two_t$ have the desired symmetries we need to be able to deduce that a property holds for all $\sigma^\two_t$ if it holds for $\mu$-almost every $t\in X$. It is not enough for us that the map $t \mapsto \sigma^\two_t$ is continuous, as we are particularly interested in $\sigma^\two_a$ and in the most challenging  situations the point $a$ does not belong to $\supp(\mu)$.
Indeed, when $a$ does belong to the support of $\mu$ one can prove \cref{thm:2-dim_cubes_exist} relatively easily as in that case one can verify $(a,a,a,a)$ is a 2-dimensional \Erdos{} cube.
Fortunately, the projections $F^* \sigma^\two_t$ only depend on the value of $\pi(t)$. We thus are able to apply the following lemma to make deductions about $\sigma^\two_t$ for \textit{all} points $t \in X$.

\begin{Lemma}
\label{lem:supp_to_cont_magic}
Let $\pi\colon(X,\mu,T)\to (Y,\nu,S)$ be a continuous factor map and assume that $\supp(\nu)=Y$.
Then the only set $F \subset X$ that is closed, has full measure and satisfies
\[
a \in F \textup{ and } \pi(b) = \pi(a) \implies b \in F
\]
is $F=X$.
\end{Lemma}
\begin{proof}
Since $F = \pi^{-1}(\pi(F))$ we conclude that the compact set $\pi(F)$ has full measure.
But $\nu$ has full support, so $\pi(F) = Y$.
\end{proof}

Recall the maps $\phi \colon  X^\two \to X^\two$ of \eqref{eqn_2-dim_permutation} and $F^*$ of \eqref{eqn:ignore_first_k2}.

\begin{Corollary}
\label{thm:sigmaka_symm_k2}
Let $(X,\mu,T)$ be an ergodic system and assume there is a continuous factor map $\pi$ to its Kronecker factor.
For every $t \in X$ the push forward $\phi \sigma^\two_t$ is equal to $\sigma^\two_t$.
\end{Corollary}
\begin{proof}
By continuity, the set
\[
H = \{ t \in X : \phi \sigma^\two_t = \sigma^\two_t \}
\]
is closed.
We wish to prove it has full measure.
To do so we  show that $t \mapsto \phi \sigma^\two_t$ is a disintegration of $\mu^\two$ and appeal to uniqueness in \cref{thm:disintegration}.
To see this, calculate
\begin{align*}
\int_X \int_{X^\two} f_{00} \otimes f_{01} \otimes f_{10} \otimes f_{11} \intd \phi \sigma^\two_t \intd \mu(t)
=
&
\int_X \int_{X^\two} f_{00} \otimes f_{10} \otimes f_{01} \otimes f_{11} \intd \sigma^\two_t \intd \mu(t)
\\
=
&
\int_{X^\two} f_{00} \otimes f_{10} \otimes f_{01} \otimes f_{11} \intd \mu^\two(t)
\end{align*}
and note that $\phi \mu^\two = \mu^\two$ by \eqref{eqn_2-dim-cubic-measure-permutation-symmetry}.
Since the map $\phi$ commutes with the projection $(x_{00},x_{01},x_{10},x_{11})\mapsto x_{00}$ onto the first coordinate we certainly have
\[
(\phi \sigma^\two_t)( \{ x \in X^\two : x_{00} = t \}) = 1
\]
for every $t \in X$.
It follows from uniqueness in \cref{thm:disintegration} that $\mu(H) = 1$.

Since $H$ has full measure, it follows that the closed set
\[
H' = \{ t \in X : F^*(\phi \sigma^\two_t) = F^*(\sigma^\two_t) \}
\]
is also of full measure.
Since $F^*$ and $\phi$ commute we conclude that whenever $t \in H'$ and $\pi(t) = \pi(s)$ one also has $s \in H'$.
Applying \cref{lem:supp_to_cont_magic} gives $H' = X$.
Since $\sigma^\two_t = \delta_t \times F^* \sigma^\two_t$ for all $t$ we are done.
\end{proof}

Our last ingredient is that $\sigma^\two_t$-almost every point $x\in X^\two$ is generic for $\lambda^\two_x$ along some \Folner{} sequence.
It is a special case of \cref{thm:sigmaka_generics}, which we state and prove in \cref{sec:cubes_exist}. 
The proofs of the two results are essentially the same, relying only on the continuity and equivariance of $t \mapsto \sigma^\two_t$ and basic facts from measure theory and ergodic theory.
To avoid repetition, we omit the proof of the following version, which is all we need for \cref{thm:2-dim_cubes_exist}, and refer the reader to the proof of \cref{thm:sigmaka_generics} in \cref{sec:cubes_exist}.

\begin{Lemma}
\label{lem:sigmaka_generics_1dim}
Let $(X,\mu,T)$ be an ergodic system and assume there is a continuous factor map to its Kronecker factor.  Fix $a \in \gen(\mu,\Phi)$ for some \Folner{} sequence $\Phi$.
Then there exists a \Folner{} sequence $\Psi$ such that for $\mu$-almost every $x\in X$ the point $(a,x)$ is generic for $\lambda^\kone_{(a,x)}$ along $\Psi$.
\end{Lemma}

We are now ready to prove \cref{thm:2-dim_cubes_exist} under the additional assumption that $(X,\mu,T)$ has a continuous factor map to its Kronecker factor $(Z,m,R)$.
In fact, all we need to do is emulate the proof of \cref{prop_a.e.-point_is_Erdos_cube_2-dim} but with $\mu^\two$ replaced by $\sigma_a^\two$.

\begin{proof}[Proof of \cref{thm:2-dim_cubes_exist}, assuming a continuous factor map]
We must produce an \Erdos{} cube in $X^\two$ whose first coordinate is $a$ and whose last coordinate belongs to $E$.

From \cref{prop_cont_kronecker_implies_cont_erg_decomp} we know each of the measures $\lambda^\kone_x$ has $\mu$ as its second marginal.
In particular $\lambda^\kone_x(X \times E) = \mu(E)$ for all $x \in X^\kone$.
It follows that
\[
\sigma_t^\two(X\times X\times X\times E)=\mu(E)>0
\]
for all $t \in X$.
It is also immediate that $\sigma^\two_t$-almost every $x \in X^\two$ has $t$ as its first coordinate.
It therefore suffices to prove that $\sigma_a^\two$-almost every point is an \Erdos{} cube, as then there must be at least one whose first coordinate equals $a$ and whose last coordinate lies in $E$.

\cref{thm:sigmaka_symm_k2} tells us that $\sigma^\two_a$ and its push forward $\phi \sigma^\two_a$ are equal.
In light of this symmetry, it suffices to show that for $\sigma_a^\two$-almost every point $x \in X^\two$ there exists an increasing sequence $c\colon\N\to\N$ such that
\[
\lim_{n\to\infty} T^{c(n)}(x_{00}, x_{01})=(x_{10}, x_{11}),
\]
because then it follows from $\phi \sigma_a^\two=\sigma_a^\two$ that $\sigma_a^\two$-almost every $(x_{00}, x_{01}, x_{10}, x_{11})$ is an \Erdos{} cube.
This is equivalent to showing that for $\mu$-almost every $x\in X$ and $\delta_{(a,x)}\times \lambda_{(a,x)}^\kone$-almost every $(x_{00}, x_{01}, x_{10}, x_{11})\in X^\two$ there is some $c\colon\N\to\N$ with
\[
\lim_{n\to\infty} T^{c(n)}(x_{00}, x_{01})=(x_{10}, x_{11})
\]
due to the description of $\sigma_a^\two$ given in \cref{prop_2-dim_sigma_rep}.
Note that the set of $x\in X$ for which $(a,x)$ is generic for $\lambda_{(a,x)}^\kone$ along some \Folner{} sequence has full measure with respect to $\mu$, due to \cref{lem:sigmaka_generics_1dim}. 
For any such $x$, the orbit of the point $(a,x)$ under $T\times T$ accumulates at any point in $\supp(\lambda_{(a,x)}^\kone)$.
Since the set $\{(a,x)\}\times \supp(\lambda_{(a,x)}^\kone)$ has full measure with respect to $\delta_{(a,x)} \times \lambda_{(a,x)}^\kone$, this finishes the proof.
\end{proof}

To complete the proof of \cref{thm:2-dim_cubes_exist}, and hence the proof of \cref{thm:main_theorem_k2}, we address the assumption that $(X,\mu,T)$ has a continuous factor map to its Kronecker factor.
In doing so, the following result is vital.

\begin{Proposition}
\label{prop:cool_point_exists_k2}
Assume $(X,\mu,T)$ is an ergodic system and $a \in X$ is transitive.
Assume $(Z,m,R)$ is a group rotation and that there is a measurable factor map $\rho$ from $(X,\mu,T)$ to $(Z,m,R)$.
Then there is a point $\tilde{z} \in Z$ and a \Folner{} sequence $\Psi$ such that
\begin{equation*}
\lim_{N \to \infty} \dfrac{1}{|\Psi_N|} \sum_{n \in \Psi_N} f(T^n a) \, g(R^n \tilde{z})
=
\int_X f \cdot (g \circ \rho) \intd \mu
\end{equation*}
for all $f \in \cont(X)$ and all $g \in \cont(Z)$.
\end{Proposition}
\begin{proof}
This is a special case of \cite[Proposition~6.1]{HK-uniformity} restated in our terminology.
\end{proof}

The existence of the point $\tilde{z}$ guaranteed by \cref{prop:cool_point_exists_k2} allows us to produce continuous factor maps to structured factors that a priori only have measurable factor maps. The analog of this  construction for $k$-fold sumsets -- in which Kronecker factors are replaced with pronilfactors -- is carried out in \cref{subsec:top_pronil}.

\begin{Proposition}
\label{prop_continuous_kronecker_factor}
Assume $(X,\mu,T)$ is an ergodic system and that $a \in \gen(\mu,\Phi)$ for some \Folner{} sequence $\Phi$ and that $a$ is transitive.
There is an ergodic system $(\tilde{X},\tilde{\mu},\tilde{T})$, a \Folner{} sequence $\Psi$, a point $\tilde{a}\in\gen(\mu,\Psi)$, and a continuous factor map $\pi_1\colon \tilde{X}\to X$ with $\pi_1(\tilde{a})=a$ such that $(\tilde{X},\tilde{\mu},\tilde{T})$ has a continuous factor map to its Kronecker factor. 
\end{Proposition}

\begin{proof}
Let $\pi\colon X\to Z$ be a measurable factor map from $(X,\mu,T)$ to its Kronecker factor $(Z,m,R)$.
Consider the transformation $\tilde{T}=T\times R$ on the product $\tilde{X}=X\times Z$.
Let $\pi_1(x,z)=x$ and $\pi_2(x,z)=z$ denote the two coordinate projections.
Let $\tilde{\mu}$ be the push forward of $\mu$ under the map $\rho\colon  X \to \tilde{X}$ defined by $\rho(x) = (x,\pi(x))$.
Note that $(X,\mu,T)$ and $(\tilde{X},\tilde{\mu},\tilde{T})$ are isomorphic via the map $\rho$ and hence have isomorphic Kronecker factors.
Thus $\pi_2$ is a continuous factor map from $(\tilde{X},\tilde{\mu},\tilde{T})$ to its Kronecker factor.
There exists, by \cref{prop:cool_point_exists_k2}, a point $\tilde{z} \in Z$ and a \Folner{} sequence $\Psi$ such that $\tilde{a}=(a,\tilde{z})$ is generic for $\tilde{\mu}$ along $\Psi$. Since $\pi_1\colon \tilde{X}\to X$ is continuous and $\pi_1(\tilde{a})=a$, we are done.
\end{proof}

\begin{proof}[Proof of \cref{thm:2-dim_cubes_exist} (in full generality)]
Suppose $(X,\mu,T)$, $E\subset X$, and $a\in X$ are as in the hypothesis of \cref{thm:2-dim_cubes_exist}.
We first replace $X$ with the closure $X'$ of $\{ T^n(a) : n \in \Z \}$.
Since $a \in \gen(\mu,\Phi)$, we have that  $\mu(X') = 1$ and $\mu(E \cap X') > 0$.
Thus $(X',\mu,T)$ is an ergodic system in which the point $a$ is transitive and the set $E' = X' \cap E$ is open in $X'$ and has positive measure. We thus proceed assuming $a$ is transitive in $(X,\mu,T)$.

 Let $(\tilde{X},\tilde{\mu},\tilde{T})$, $\pi_1\colon \tilde{X}\to X$, and $\tilde{a}\in\tilde{X}$ be as guaranteed by \cref{prop_continuous_kronecker_factor}. Define $\tilde{E}=\pi_1^{-1}(E)$. Since $(\tilde{X},\tilde{\mu},\tilde{T})$ admits a continuous factor map onto its Kronecker factor and \cref{thm:2-dim_cubes_exist} has already been proven for such systems, there exists a $2$-dimensional \Erdos{} cube $\tilde{x} \in \tilde{X}^\two$ with $\tilde{x}_{00}=\tilde{a}$ and $\tilde{x}_{11}\in \tilde{E}$. Since $\pi_1\colon \tilde{X}\to X$ is continuous, the quadruple
\[
x = \bigl(\pi_1(\tilde{x}_{00}),\pi_1(\tilde{x}_{01}),\pi_1(\tilde{x}_{10}),\pi_1(\tilde{x}_{11})\bigr)
\]
is an \Erdos{} cube in $X^\two$ with $x_{00}=a$ and $x_{11}\in E$, finishing the proof.
\end{proof}

\section{\Erdos{} cubes}
\label{sec:erdos-cubes}
In the previous section we defined $2$-dimensional \Erdos{} cubes, which played via \cref{thm:2-dim_cubes_exist}  a vital role in our proof of the $k=2$ case of \cref{thm:main_theorem}.
The purpose of this section is
to generalize the notion of \Erdos{} cubes to higher dimensions. We then formulate the main dynamical result of this paper, \cref{thm:cubes_exist}, which is a natural extension of \cref{thm:2-dim_cubes_exist}. At the end of this section we prove that \cref{thm:cubes_exist} implies \cref{thm:main_theorem}.

\subsection{Defining \Erdos{} cubes}
\label{subsec:defining_erdos}

For $k\in\N$, define the \define{$k$-dimensional cube} to be $\{0,1\}^k$ and denote it by $\k$. 
An element of $\k$ is written as $\epsilon = \epsilon_1 \cdots \epsilon_k$, usually without any commas or parentheses other than when this may create ambiguity.
We enumerate the elements of $\k$ in lexicographic order. For our purposes, the most relevant feature of the ordering is that the first element is $\vec 0 = 0 \cdots 0$ and the last element is $\vec 1 = 1 \cdots 1$.  For example, the elements of $\three$ listed in order are: $000, 001, 010, 011, 100, 101, 110, 111$.

For any set $X$, a point $x\in X^\k$ is written as $ x = (x_\epsilon : \epsilon\in\k)$, again enumerating the indices in lexicographic order.
For clarity, when we write $x\in X^\k$ as a $2^k$-tuple indexed by $\k$, we separate the entries with commas. 
By convention, when we write $X^{\llbracket 0 \rrbracket}$, we mean just $X$.

A \define{face} of $\k$ is a subset of $\k$ determined by fixing a single coordinate.
These play a particular role in our analysis, making use of the elementary observation that the complement of such a face is also a face, and we refer to this as the \define{opposite face}.
There are $k$ pairs of opposite faces in $\k$.
For each $1 \le i \le k$ write $F^i x$ for the $i$th upper face of $x \in X^\k$ and $F_i x$ for the $i$th lower face of $x$.
Formally, these belong to $X^\kdown$ and are defined for $x \in X^\k$ as follows:
\begin{gather*}
F^i x = (x_\epsilon :  \epsilon_i = 1 )\quad \text{ (\define{i\textsuperscript{th}  upper face of }} x) \\
F_i x = (x_\epsilon : \epsilon_i = 0 ) \quad \text{ (\define{i\textsuperscript{th} lower face of }} x)
\end{gather*}
Write also $F^* x$ for the point
\begin{equation}
\label{eqn_fstar}
F^* x = (x_\epsilon : \epsilon \ne \vec 0)
\end{equation}
in $X^{\k \setminus \{ \vec 0 \}}$.

For example, if $k=2$ then a point $x=(x_{00},x_{01},x_{10},x_{11})\in X^\two$ has two lower faces, $F_1x=(x_{00},x_{01})$ and $F_2x=(x_{00},x_{10})$, and two upper faces, $F^1x=(x_{10},x_{11})$ and $F^2x=(x_{01},x_{11})$. We also have $F^*(x) = (x_{01},x_{10},x_{11})$.
When $k=3$ the lower faces of 
\[
x=(x_{000},x_{001},x_{010},x_{011},x_{100},x_{101},x_{110},x_{111})\in X^\three
\]
are
\begin{align*}
F_1x &= (x_{000},x_{001},x_{010},x_{011}) \\
F_2x &= (x_{000},x_{001},x_{100},x_{101}) \\
F_3x &= (x_{000},x_{010},x_{100},x_{110})
\end{align*}
and the corresponding upper faces are
\begin{align*}
F^1x &= (x_{100},x_{101},x_{110},x_{111}) \\
F^2x &= (x_{010},x_{011},x_{110},x_{111}) \\
F^3x &= (x_{001},x_{011},x_{101},x_{111})
\end{align*}
respectively.
As a visual aid, we can think of a point in $X^\three$ as being represented by a $3$-dimensional cube:
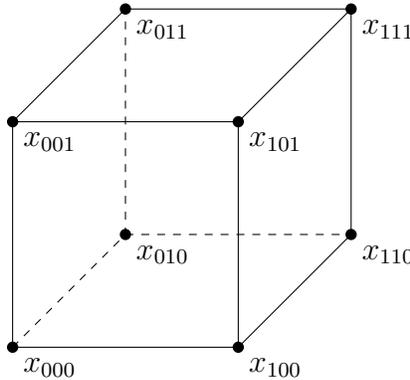
\[
\begin{tikzpicture}
\coordinate (a) at (0,0);
\coordinate (y) at (0,3);
\coordinate (w) at (1.5,1.5);
\coordinate (z) at (1.5,4.5);
\coordinate (p) at (3,0);
\coordinate (q) at (3,3);
\coordinate (r) at (4.5,1.5);
\coordinate (s) at (4.5,4.5);
\draw (a) -- (y) -- (z) -- (s) -- (r) -- (p) -- (a);
\draw[dashed] (a) -- (w) -- (z);
\draw (p) -- (q) -- (y);
\draw[dashed] (w) -- (r);
\draw (q) -- (s);
\fill (a) circle (0.075);
\node[anchor=north west] at (a) {$x_{000}$};
\fill (y) circle (0.075);
\node[anchor=north west] at (y) {$x_{001}$};
\fill (w) circle (0.075);
\node[anchor=north west] at (w) {$x_{010}$};
\fill (z) circle (0.075);
\node[anchor=north west] at (z) {$x_{011}$};
\fill (p) circle (0.075);
\node[anchor=north west] at (p) {$x_{100}$};
\fill (q) circle (0.075);
\node[anchor=north west] at (q) {$x_{101}$};
\fill (r) circle (0.075);
\node[anchor=north west] at (r) {$x_{110}$};
\fill (s) circle (0.075);
\node[anchor=north west] at (s) {$x_{111}$};
\end{tikzpicture}
\]
This geometric interpretation of points in $X^\three$ motivates our use of terminology.

Whenever $T$ is a homeomorphism of $X$, we write $T^\k$ for the homeomorphism of $X^\k$ defined by
\[
(T^\k x)_\epsilon = T(x_\epsilon)
\]
for all $x \in X^\k$ and $\epsilon\in\k$.
Note that if $x \in X^\k$ we can apply $T^\kdown$ to any of the upper and lower faces of $x$.

A permutation $\phi$ of $\{1,\dots,k\}$ induces a map $(\phi \epsilon)_i = \epsilon_{\phi(i)}$ from $\k$ to $\k$ that, by an abuse of notation, we also denote by $\phi$.
With a further abuse of notation we also use $\phi$ to denote the map $X^\k\to X^\k$ it induces via
\begin{equation}
\label{eqn:perm_on_measures}
(\phi x)_\epsilon = x_{\phi(\epsilon)}
\end{equation}
for all $\epsilon \in \k$.
See \eqref{eqn_2-dim_permutation} for an example of a two dimensional permutation.
For each $1 \le i \le k$ write $\phi_i$ for the permutation that exchanges $1$ with $i$ and fixes all other points. One then has
\begin{equation}
\label{eqn:perms_changing_faces}
F^1(\phi_i x) = F^i x \quad \text{ and } \quad F_1(\phi_i x) = F_i x
\end{equation}
for all $1 \le i \le k$. The maps $\phi$ are referred to as  \define{permutations of digits} in~\cite[Chapter 6]{HK-book}.

Mimicking the case $k = 2$, a $k$-dimensional \Erdos{} cube is a point $x \in X^\k$ thought of as forming the vertices of a $k$-dimensional cube, with the property that its lower faces dynamically approximate their corresponding upper faces.
The formal definition is as follows.

\begin{Definition}[\Erdos{} Cubes]
\label{def:erdos_cubes}
Fix a compact metric space $X$ and a homeomorphism $T$ of $X$.
A point $x \in X^\k$ is a \define{$k$-dimensional \Erdos{} cube} if
\begin{equation}
\label{eqn:erdos_cube_approximating_def}
F^i x \in \omega( F_i x, T^\kdown)
\end{equation}
for each $1 \le i \le k$.
\end{Definition}

Note that \eqref{eqn:erdos_cube_approximating_def} simply says there exists a strictly increasing sequence $c_i\colon\N\to\N$ such that
\[
\lim_{n \to \infty} (T^\kdown)^{c_i(n)} (F_i x) = F^i x
\]
holds.
Even though we are working with an invertible map $T\colon X \to X$, we require in the definition of an \Erdos{} cube that $F^i x$ is an accumulation point of the \textit{forward} iterates of $F_i x$ under $T^\kdown$.

We refer to them as cubes because of their connection to cube systems and since they form a natural subclass of the \define{dynamical cubes} (see \cref{def_dyn_cubes}) introduced in~\cite{HKM} in the study of topological pronilfactors. 
In \cref{subsec:comparing_cubes} we describe the relations among various notions of cubes, including examples and non-examples of these objects.

The next two lemmas are useful properties that follow immediately from the definitions of \Erdos{} cubes and the permutations in~\eqref{eqn:perms_changing_faces}.

\begin{Lemma}
\label{lemma_erdos_cubes_symmetries}
Let $(X,T)$ be a topological system and let $k \in \N$.
A point $x \in X^\k$ is a $k$-dimensional \Erdos{} cube if and only if
\[
\phi_i x \in \{ y \in X^\k : F^1 y \in \omega(F_1 y, T^\kdown) \}
\]
for all $1 \le i \le k$.
\end{Lemma}

\begin{Lemma}
\label{lem:erdos_face_is_erdos}
Assume $(X,T)$ is a topological system and that $x \in X^\k$ is a $k$-dimensional \Erdos{} cube.
Then $F^i(x)$ and $F_i(x)$ are $(k-1)$-dimensional \Erdos{} cubes for all $1 \le i \le k$.
\end{Lemma}

The following theorem is the main dynamical result of this paper. It extends \cref{thm:2-dim_cubes_exist} and, as we show in \cref{sec:building_sumsets},  implies our main combinatorial result \cref{thm:main_theorem} for multiple sumsets.

\begin{Theorem}
\label{thm:cubes_exist}
Assume $(X,\mu,T)$ is an ergodic system, that $a \in \gen(\mu,\Phi)$ for some \Folner{} sequence $\Phi$ in $\N$, and that $E\subset X$ satisfies $\mu(E)>0$.
For every $k \in \N$ there is a $k$-dimensional \Erdos{} cube $x \in X^\k$ with $x_{\vec0}=a$ and $x_{\vec1}\in E$.
\end{Theorem}

\subsection{\cref{thm:cubes_exist} implies \cref{thm:main_theorem}}
%\subsection{\texorpdfstring{\cref{thm:cubes_exist} implies \cref{thm:main_theorem}}{Theorem 4.4 implies Theorem 1.1}}
\label{sec:building_sumsets}

In this subsection we use \cref{thm:cubes_exist} to prove \cref{thm:main_theorem}. 
Our method streamlines and generalizes the proof that \cref{thm:2-dim_cubes_exist} implies \cref{thm:main_theorem_k2} that we have seen in \cref{sec_B+C}.
We need to set up some terminology. Given $\epsilon=(\epsilon_1,\dots,\epsilon_k)\in\k$, denote by $\bar\epsilon\in\k$ the complement $(1-\epsilon_1,\dots,1-\epsilon_k)$.
Let $x\in X^\k$ be an \Erdos{} cube and let $E\subset X$ be an open set such that $x_{\vec1}\in E$. 
Let $B_1,\dots,B_k\subset\N$. 
We say that $(B_1,\dots,B_k)$ is \emph{$(x,E)$-acceptable} if for any $b=(b_1,\ldots,b_k)\in B_1\times\cdots\times B_k$ and any $\epsilon=(\epsilon_1,\ldots,\epsilon_k)\in\k$, 
\begin{equation}
\label{eq_acceptable}
T^{\bar\epsilon\cdot b}x_{\epsilon}\in E,
\end{equation}
where $\bar\epsilon\cdot b= \bar\epsilon_1b_1+\cdots+\bar\epsilon_kb_k$.
Note in particular that
\[
B_1+\cdots+B_k\subset\{n:T^nx_{\vec 0}\in E\}
\]
whenever $(B_1,\dots,B_k)$ is $(x,E)$-acceptable by taking $\epsilon = {\vec 0}$ in \eqref{eq_acceptable}.
To prove \cref{thm:main_theorem}, we make use of the following lemma as an iterative step.

\begin{Lemma}
Let $k\in\N$, let $(X,T)$ be a topological dynamical system, let $x\in X^\k$ be an \Erdos{} cube, and let $E\subset X$ be an open set with $x_{\vec 1}\in E$.
Assume $(B_1,\dots,B_k)$ consists of finite sets and is $(x,E)$-acceptable.
Fix $i \in \{1,\dots,k\}$.
There exists $c\in\N\setminus B_i$ such that, letting $B_i'=B_i\cup\{c\}$ and $B_j'=B_j$ for all $j\neq i$, the tuple $\big(B_1',\dots,B_k'\big)$ is still $(x,E)$-acceptable.
\end{Lemma}
\begin{proof}
Fix $i \in \{1,\dots,k\}$.
Given $\epsilon\in\k$ with $\epsilon_i=0$, let $ \eta(\epsilon)\in\k$ be the element that differs from $\epsilon$ only on the $i$th coordinate.
Define
\[
H_\epsilon
=
\bigcap_{b \in B_1 \times \cdots \times B_k}T^{-\overline{\eta(\epsilon)}\cdot b} E
\]
for each $\epsilon \in \k$ with $\epsilon_i = 0$.
In view of \eqref{eq_acceptable}, we have
that $x_{ \eta(\epsilon)} \in H_\epsilon$ whenever $\epsilon_i = 0$. 
Since $x$ is an \Erdos{} cube, $F^ix\in\omega(F_ix,T^\kdown)$, and so in particular there exists $c\in\N\setminus B_i$ such that 
\begin{equation}
\label{eq_messyproof}
\text{ for all } \epsilon\in\k, \ \epsilon_i = 0 \Rightarrow T^c x_\epsilon \in H_\epsilon.
\end{equation}
We claim that with this choice of $c$ the tuple  $(B_1',\dots,B_k' )$ in the statement of the lemma is $(x,E)$-acceptable.
Let $\epsilon\in\k$ and $b\in B_1'\times\cdots\times B_k'$.
If $\epsilon_i=1$ or if $b_i\neq c$, then \eqref{eq_acceptable} holds automatically.
If $\epsilon_i=0$ and $b_i=c$, then \eqref{eq_messyproof} and the definition of $H_\epsilon$ imply $T^c x_\epsilon \in T^{-\overline{ \eta(\epsilon)}\cdot b}E$ as the $i$th coordinate of $\overline{ \eta(\epsilon)}$ is $0$.
In other words, $T^{c+\overline{ \eta(\epsilon)}\cdot b}x_\epsilon\in E$, but since $c+\overline{ \eta(\epsilon)}\cdot b=\bar \epsilon\cdot b$, this is precisely \eqref{eq_acceptable}.
\end{proof}

\begin{Corollary}
\label{corollary_infiniteacceptable}
Let $(X,T)$ be a topological dynamical system, let $x\in X^\k$ be an Erd\H os cube,  and let $E\subset X$ be an open set with $x_{\vec 1}\in E$.
Then there exist infinite sets $B_1,\dots,B_k\subset\N$ such that $(B_1,\dots,B_k)$ is $(x,E)$-acceptable.
\end{Corollary}
\begin{proof}
The tuple $(\emptyset,\dots,\emptyset)$ is vacuously $(x,E)$-acceptable for any $x,E$.
Iterating the previous lemma gives the corollary.
\end{proof}

We are now ready to prove \cref{thm:main_theorem}.
\begin{proof}[Proof of \cref{thm:main_theorem}]
Let $A\subset\N$ have positive Banach upper density.
By \cref{thm_fcp}, there exists an ergodic system $(X,\mu,T)$,
a point $a \in X$, 
a clopen set $E\subset X$, 
and a  \Folner{} sequence $\Phi$
such that $\mu(E) > 0$, $A = \{n \in \N : T^n a \in E\}$, and $a \in \gen(\mu,\Phi)$.
By \cref{thm:cubes_exist}, there exists an \Erdos{} cube $x \in X^\k$ with $x_{\vec 0}=a$ and $x_{\vec 1} \in E$.
By \cref{corollary_infiniteacceptable},  there exist infinite sets $B_1,\dots,B_k\subset\N$ such that $(B_1,\dots,B_k)$ is $(x,E)$-acceptable.
Thus
\[
B_1 + \cdots + B_k \subset \{ n \in \N : T^n x_{\vec 0}\in E \} = A
\]
by taking $\epsilon=\vec0$ in \eqref{eq_acceptable}.
\end{proof}

\section{Cube systems and pronilfactors}
\label{sec:cubism}

Our main result - \cref{thm:cubes_exist} - is that \Erdos{} cubes with prescribed first coordinate exist under certain assumptions.
In fact, we are only able to prove this directly when the system $(X,\mu,T)$ has continuous factor maps to certain structured factors known as pronilfactors.
The proof of \cref{thm:cubes_exist} then follows in general from the fact that every ergodic system $(X,\mu,T)$ has an extension with such factor maps.
In this section we recall what we need from the theory of pronilfactors and prove that every ergodic system has an extension as described above.
We conclude with some discussion on the relationship between \Erdos{} cubes and other types of cubes in the dynamics literature.
The material of \cref{subsec:comparing_cubes} is not needed for the proof of \cref{thm:cubes_exist}.

\subsection{Cubic measures and structured factors}
\label{sec:cubic-measures}
Let $(X,\mu,T)$ be an ergodic system. 
It was important in \cref{sec_B+C} to have an ergodic decomposition of $\mu^\kone = \mu \times \mu$.
As discussed in the introduction, the appropriate generalization of $\mu \times \mu$ for the proof of \cref{thm:main_theorem} in general are the cubic measures $\mu^\k$ whose definition we recall.
We remark that while we refer to these measures and other objects as various types of cubes and use associated geometric terminology, such as faces or sides, these objects should more properly be referred to as parallelograms in two dimensions or as parallelepipeds in higher dimensions.  

For each $k\in\N$, the \define{$k$-dimensional cube   system} $(X^\k, \mu^\k, T^\k)$ is the system defined on the product space $X^\k$ endowed with the transformation $T^\k = T \times \cdots \times T$, the product taken $2^k$ times, and equipped with the measure $\mu^\k$ defined inductively by 
\begin{align}
\label{eqn_def_cubic_measures}
\mu^\kup
=
\int_{X^\k} (\mu^\k)_x \times (\mu^\k)_x \intd \mu^\k(x)
\end{align}
where $x\mapsto (\mu^{\k})_x$ is an ergodic decomposition of $\mu^{\k}$.
Note that
\[
(X^\kone, \mu^\kone, T^\kone) = (X \times X, \mu \times \mu, T \times T)
\]
by our assumption that $(X,\mu,T)$ is ergodic, but when $k> 1$ the measure $\mu^\k$ is generally not equal to the product measure $\mu^\kdown \times \mu^\kdown$.
The following lemma states that $\mu^\k$ is invariant under the maps $\phi\colon  X^\k \to X^\k$ induced by permutations $\phi$ of $\{1,\dots,k\}$ as defined in \eqref{eqn:perm_on_measures}.

\begin{Lemma}[{\cite[Proposition~8, Chapter~8]{HK-book}}]
\label{lemma_muk_permutation_invariant}
Fix $k \in \N$.
For every permutation $\phi$ of $\{1,\dots,k\}$ the push forward $\phi \mu^\k$ is equal to $\mu^\k$.
\end{Lemma}

We also need later the following description of $\mu^\k$.

\begin{Lemma}
\label{lem:dont_know_if_obvious}
For every $k \in \N$, we have
\begin{align}
\label{eqn_alt_def_cubic_measures}
\mu^\kup = \int_{X^\k} \delta_x \times (\mu^\k)_x \intd \mu^\k(x).
\end{align}
\end{Lemma}
\begin{proof}
Using the ergodic decomposition
\[
\mu^\k = \int_{X^\k} (\mu^\k)_x \intd \mu^\k(x)
\]
we have that
\[
\int_{X^\k} \delta_x \times (\mu^\k)_x \intd \mu^\k(x)
=
\int_{X^\k} \int_{X^\k} \delta_y \times (\mu^\k)_y \intd \mu^\k_x(y) \intd \mu^\k(x).
\]
Since $x\mapsto (\mu^\k)_x$ is the ergodic decomposition of $\mu^\k$, for $\mu^\k$-almost every $x\in X^\k$ and $(\mu^\k)_x$-almost every $y\in X^\k$, we have that  $(\mu^\k)_y = (\mu^\k)_x$. 
Thus, for $\mu^\k$-almost every $x\in X^\k$ 
it follows that 
\[
\int_{X^\k} \delta_y \times (\mu^\k)_y \intd (\mu^\k)_x(y)
=
\int_{X^\k} \delta_y \times (\mu^\k)_x \intd (\mu^\k)_x(y)
=
(\mu^\k)_x \times (\mu^\k)_x
\]
and so~\eqref{eqn_def_cubic_measures} and~\eqref{eqn_alt_def_cubic_measures} 
define the same measure.
\end{proof}

Our proof of \cref{prop_cont_kronecker_implies_cont_erg_decomp} needs the system $(X,\mu,T)$ to have a continuous factor map to its Kronecker factor, and it is essential in the proof that functions on the Kronecker factor describe the invariant $\sigma$-algebra of the system $(X^\kone,\mu^\kone,T^\kone)$. 
The appropriate generalization of Kronecker factor is the $k$-step pronilfactor of $(X,\mu,T)$ and the structure theory of measure preserving systems  tells us that functions on pronilfactors can be used to describe the invariant $\sigma$-algebra of $(X^\k,\mu^\k,T^\k)$. 
To formally state this in the form we need, we recall some details about nilsystems.

If $G$ is a $k$-step nilpotent Lie group and $\Gamma$ is a discrete cocompact subgroup of $G$, then the compact manifold $X = G/\Gamma$ is a \define{$k$-step nilmanifold}.  The group $G$ acts on $X$ by left translation, and there is a unique Borel probability measure $\mu$ on $X$ that is invariant under this action, called the \define{Haar measure}. 

If $g\in G$ is a fixed element, $T\colon X\to X$ is the map $x\mapsto gx$, and $\mu$ is the Haar measure, then $(X, \mu, T)$
is  a \define{ $k$-step nilsystem}.
It is a classical result that for all $k$-step nilsystems $(X,\mu, T)$ the properties (i) minimal; (ii) transitive; (iii) ergodic; (iv) uniquely ergodic; are equivalent.
(See~\cite[Theorem~11, Chapter~11]{HK-book} for a summary and discussion.) 

The connection between $\mu^\k$ and $k$-step nilsystems is phrased in terms of certain seminorms $\ghk \cdot \ghk_{k+1}$ on $\lp^\infty(X,\mu)$ introduced by Host and Kra~\cite{HK-05}. For our purposes we only need the connection below between these seminorms and the measure $\mu^\k$. The theory of these seminorms can be found in~\cite{HK-book}.
To fully explain the connection we need to discuss inverse limits of systems.
Suppose for each $j\in\N$ we have a system $(X_j,\mu_j,T_j)$, and suppose there are continuous factor maps $\pi_j \colon X_{j} \to X_{j-1}$ for all $j>1$.
Then there exists a unique \emph{inverse limit} $(X,\mu,T)$ which is a system admitting continuous factor maps $\psi_j \colon X \to X_j$ such that $\pi_j\circ\psi_j=\psi_{j-1}$ for every $j>1$, and such that 
\[
\bigcup_{j\in\N}C(X_j)\circ\psi_j
\]
is dense in $C(X)$.
Additionally, whenever each of the systems $(X_j,\mu_j,T_j)$ is ergodic, so is the inverse limit $(X,\mu,T)$.
(See for example {\cite[Section~6, Chapter 3]{HK-book}} for the details.)

\begin{Theorem}
\label{thm:pronilfactors}
Fix an ergodic system $(X,\mu,T)$.
For every $k \in \N$ there is a system $(Z_k,m_k,T)$ with all the following properties.
\begin{itemize}
\item
The system $(Z_k,m_k,T)$ is an inverse limit of ergodic $k$-step nilsystems.
\item
There is a measurable factor map $\rho_k$ from $(X,\mu,T)$ to $(Z_k,m_k,T)$.
\item
For any $f \in \lp^\infty(X,\mu)$ we have $\ghk f \ghk_{k+1} = 0$ if and only if $f$ and $\lp^\infty(Z_k,m_k) \circ \rho_k$ are orthogonal in $\lp^2(X,\mu)$.
\end{itemize}
\end{Theorem}
\begin{proof}
This is a combination of Theorem~2 in Chapter~16, Theorem~7 in Chapter~9, and the discussion at the beginning of Section 3.1 in Chapter 13 of~\cite{HK-book}.
\end{proof}

\begin{Definition}
\label{def:pronil}
Whenever $(X,\mu,T)$ is ergodic, we call the system $(Z_k,m_k,T)$ satisfying the conclusion of \cref{thm:pronilfactors} the $k$-step \define{pronilfactor} of $(X,\mu,T)$. 
\end{Definition}

The role of the $k$-step pronilfactor in the proof of \cref{thm:cubes_exist} is analogous to the role of the Kronecker factor in the proof of  \cref{thm:2-dim_cubes_exist}.

\subsection{Continuous factor maps to pronilfactors}
\label{subsec:top_pronil}

It is crucial for the proof of \cref{prop_cont_kronecker_implies_cont_erg_decomp} that the system $(X,\mu,T)$ has a continuous factor map to its Kronecker factor. In this section we explain how to arrive at this situation via an extension procedure that holds in the more general setting of pronilfactors.

\begin{Definition}
\label{def:top_pro_nil}
Given an ergodic system $(X, \mu, T)$ and $k\in\N$, let $(Z_k,m_k,T)$ be its $k$-step pronilfactor.
We say that $(X, \mu, T)$ \define{has topological pronilfactors} if for every $k\in\N$ there is a continuous factor map $\pi_k\colon X\to Z_k$.
\end{Definition}

\begin{Remark}
It follows from a careful analysis of our proof of \cref{thm:main_theorem} that, to find $k$-dimensional \Erdos{} cubes, it suffices to have a continuous factor map to the $k$-step pronilfactor. However, there is no loss in assuming that we have topological pronilfactors for all $k\in\N$, and this simplifies the terminology.
\end{Remark}

Note that not every system has topological pronilfactors. In fact, there are systems (see~\cite{Lehrer} or~\cite{DFM}) with non-trivial Kronecker factor but no continuous eigenfunctions, which precludes the existence of a continuous factor map to a non-trivial Kronecker system.
We handle such systems using the following proposition,  which allows us to associate with every ergodic system a larger system that has topological pronilfactors and which plays a role analogous to that played by \cref{prop:cool_point_exists_k2} in \cref{sec_B+C}.
\begin{Proposition}
\label{prop:cool_point_exists}
Assume $(X,\mu,T)$ is an ergodic system and $a \in X$ is transitive.
Assume $(Z,S)$ is a topological system that is distal and that $m$ is an $S$-invariant measure on $Z$.
Assume also that there is a measurable factor map $\rho$ from $(X,\mu,T)$ to $(Z,m,S)$.
Then there is a point $\tilde{z} \in Z$ and a \Folner{} sequence $\Psi$ such that
\begin{equation}
\label{eq_distalgenericfactor}
\lim_{N \to \infty} \dfrac{1}{|\Psi_N|} \sum_{n \in \Phi_N} f(T^n a) \, g(S^n \tilde{z})
=
\int_X f \cdot (g \circ \rho) \intd \mu
\end{equation}
for all $f \in \cont(X)$ and all $g \in \cont(Z)$.
\end{Proposition}
\begin{proof}
This is a restatement of \cite[Proposition~6.1]{HK-uniformity} in our terminology.
\end{proof}

\begin{lemma}
\label{lem:nice_extension}
Fix an ergodic system $(X,\mu,T)$ and a transitive point $a \in X$.
There is an ergodic system $(\tilde{X},\tilde{\mu},\tilde{T})$ with topological pronilfactors, a \Folner{} sequence $\Psi$, a point $\tilde{a} \in \gen(\tilde{\mu},\Psi)$ and a continuous factor map $\pi \colon \tilde{X} \to X$ with $\pi(\tilde{a}) = a$.
Moreover $(X,\mu,T)$ and $(\tilde{X},\tilde{\mu},\tilde{T})$ are measure theoretically isomorphic systems.
\end{lemma}

\begin{proof}
For each $k\in\N$, we apply \cref{thm:pronilfactors} to $(X,\mu,T)$ 
to obtain the $k$-step pronilfactor $(Z_k, m_k, T)$ of $(X,\mu,T)$.
It is an inverse limit of ergodic $k$-step nilsystems and there is a measurable factor maps $\rho_k \colon (X, \mu, T)\to (Z_k, m_k, T)$.
Define $(Z,m,S)$ to be the system where 
\[
Z = \prod_{k \in \N} Z_k, 
\]
$m$ is the push forward of $\mu$ under the measurable map $\rho$ defined by
\[
\rho(x) = (\rho_k(x) : k \in \N)
\]
for all $x \in X$, and $S$ is the product transformation. 
Note that $S$ is distal on $Z$ and that $\rho$ is a measurable factor map from $(X,\mu,T)$ to $(Z,m,S)$.
We can therefore apply \cref{prop:cool_point_exists}
to obtain some point $\tilde{z} \in Z$ and a \Folner{} sequence $\Psi$ such that \eqref{eq_distalgenericfactor} holds for all $f \in \cont(X)$ and all $g \in \cont(Z)$.
Write $\tilde{a} = (a,\tilde{z})$ and $\tilde{T} = T \times S$.
Note that \eqref{eq_distalgenericfactor} implies $\tilde{a} \in \gen(\tilde{\mu},\Psi)$, where $\tilde{\mu}$ is the push forward of $\mu$ under the map $x \mapsto (x,\rho(x))$.
Take $\tilde{X}$ to be the orbit closure of $\tilde{a}$ in $X \times Z$ 
and let $\pi \colon \tilde{X} \to X$ be the projection on the first coordinate.

The maps $x \mapsto (x,\rho(x))$ and $(x,z) \mapsto x$ are almost sure inverses of one another and therefore the systems $(X,\mu,T)$ and $(\tilde{X},\tilde{\mu},\tilde{T})$ are measure theoretically isomorphic. In particular, it follows that $(\tilde{X},\tilde{\mu},\tilde{T})$ is ergodic and that its $k$-step pronilfactor is isomorphic to $(Z_k,m_k,T)$ for each $k \in \N$. Thus the map $(x,z) \mapsto z_k$ is a continuous factor map from $(\tilde{X},\tilde{\mu},\tilde{T})$ to its $k$-step pronilfactor.
\end{proof}

The upshot is that the following theorem implies \cref{thm:cubes_exist}.

\begin{Theorem}
\label{thm:cubes_exist_pronil}
Assume $(X,\mu,T)$ is an ergodic system with topological pronilfactors, that $a \in \gen(\mu,\Phi)$ for some \Folner{} sequence $\Phi$ and that $E \subset X$ satisfies $\mu(E) > 0$. 
For every $k \in \N$ there is a $k$-dimensional \Erdos{} cube $x \in X^\k$ with $x_{\vec0}=a$ and $x_{\vec 1} \in E$.
\end{Theorem}

\begin{proof}[Proof that \cref{thm:cubes_exist_pronil} implies \cref{thm:cubes_exist}]
Assume $(X,\mu,T)$ is an ergodic system, that $a \in \gen(\mu,\Phi)$ for some \Folner{} sequence $\Phi$, and that $E \subset X$ has positive measure.
Fix $k \in \N$.
We wish to apply \cref{lem:nice_extension} and for this we need the point $a$ to be transitive.
We therefore replace $X$ with the closure $X'$ of $\{ T^n(a) : n \in \Z \}$.
Since $a \in \gen(\mu,\Phi)$, we must have $\mu(X') = 1$ and $\mu(E \cap X') > 0$.
Thus $(X',\mu,T)$ is an ergodic system in which the point $a$ is transitive and the set $X' \cap E$ has positive measure.
Apply \cref{lem:nice_extension} to $(X',\mu,T)$ to get a system $(\tilde{X},\tilde{\mu},\tilde{T})$, a \Folner{} sequence $\Psi$, a point $\tilde{a} \in \gen(\tilde{\mu},\Psi)$, and a continuous factor map $\pi$ from $(\tilde{X},\tilde{\mu},\tilde{T})$ to $(X',\mu,T)$.
Set $\tilde{E} = \pi^{-1}(X' \cap E)$.

Now apply \cref{thm:cubes_exist_pronil} to get an \Erdos{} cube $\tilde{x} \in \tilde{X}^\k$ with $\tilde{x}_{\vec 0}  = \tilde{a}$ and $\tilde{x}_{\vec 1} \in E$.
It is immediate that $\pi(\tilde{x})$ is an \Erdos{} cube whose first coordinate is $a$ and whose last coordinate belongs to $E$.
\end{proof}

\subsection{Nilcubes}
We now turn to an issue in the proof of \cref{thm:main_theorem} for $k \ge 3$ that is not apparent in the case $k=2$ discussed in \cref{sec_B+C}.
To prove \cref{prop_cont_kronecker_implies_cont_erg_decomp} we defined for \textit{every} point $x \in X^\kone$ a measure $\lambda^\kone_x$ on $X \times X$ with various properties.
In the next section we define measures $\lambda^\k_x$ on $X^\k$ to prove \cref{thm:main_theorem} in general.
However, even when $(X,\mu,T)$ has topological pronilfactors, we are not able to define $\lambda^\k_x$ for all $x \in X^\k$.
In preparation for the next section, we define the space that is the domain of the map $x \mapsto \lambda^\k_x$.
To do so, we recall the definition of dynamical cubes from~\cite{HKM}.

\begin{Definition}[Dynamical Cubes]
\label{def_dyn_cubes}
For a topological system $(X, T)$, we define the set of \define{$k$-dimensional dynamical cubes} $\cube^\k(X,T)$ to be the orbit closure of the diagonal points under the face transformations.
Formally
\[
\cube^\k(X, T) = \overline{ \{ ( T^{\epsilon \cdot n} x : \epsilon \in \k ) : x \in X  \textup{ and } n \in \Z^k \} }
\]
with the closure taken in $X^\k$ and $\epsilon \cdot n = \epsilon_1 n_1 + \cdots + \epsilon_k n_k$ for $\epsilon \in \k$ and $n \in \Z^k$.
\end{Definition}

In our setting we only have a single transformation on a given topological system, so we write $\cube^\k(X)$ instead of the more cumbersome $\cube^\k(X,T)$. For example, $\cube^\two(X)$ is the closure of
\[
\{ (x, T^n x, T^m x, T^{m+n} x) : m,n \in \Z \textup{ and } x \in X \}
\]
in $X^\two$.

\begin{Theorem}[{\cite[Theorem 8, Chapter 12]{HK-book}}]
\label{thm:nil_cubes_is_nil}
Assume $(Z,m,T)$ is an ergodic $k$-step nilsystem.
Then $(\cube^\k(Z),T^\k,m^\k)$ is a $k$-step nilsystem, and in particular $m^\k(\cube^\k(Z))=1$.
\end{Theorem}

\begin{Definition}[Nilpotent Cubes]
For an ergodic system $(X,\mu,T)$ with topological pronilfactors, we define the set of \define{$k$-dimensional nilcubes} to be
\[
\nilc^\k(X) = (\pi_k^\k)^{-1} (\cube^\k(Z_k))
\]
where $(Z_k,m_k,T)$ is the $k$-step pronilfactor of $(X,\mu,T)$ and $\pi_k\colon X\to Z_k$ is the associated continuous factor map.
\end{Definition}

It is the set $\nilc^\k(X)$ that plays the role of the domain for the map $x \mapsto \lambda^\k_x$ we introduce in \cref{sec:cont_erg_decomp}.
The following basic properties of $\nilc^\k(X)$ are used frequently.

\begin{lemma}
\label{lem:transitive_qone_full}
Assume the ergodic system $(X,\mu,T)$ has topological pronilfactors.
Then $\nilc^\kone(X) = X^\kone$.
\end{lemma}
\begin{proof}
Fix a continuous factor map $\pi$ from $(X,\mu,T)$ to its Kronecker factor $(Z,m,T)$.
Since the Kronecker factor is an ergodic group rotation every point therein has dense orbit, and $\cube^\kone(Z)  = Z^\kone$. Thus $\nilc^\kone(X)=(\pi^\kone)^{-1}(\cube^\kone(Z_k)=(\pi^\kone)^{-1}(Z^\kone)=X^\kone$.
\end{proof}

\begin{lemma}
\label{lem_nilcprops}
Assume the ergodic system $(X,\mu,T)$ has topological pronilfactors. Then $\nilc^\k(X)$ is closed,
\begin{equation}
\label{eq_cubescontainedinnilcubes}
\cube^\k(X) \subset \nilc^\k(X)
\end{equation}
for all $k \in \N$, and $\mu^\k(\nilc^\k(X)) = 1$.
\end{lemma}

\begin{proof}
Fix $k \in \N$.
Write $\pi$ for the continuous factor map from $(X,\mu,T)$ to its $k$-step pronilfactor $(Z,m,T)$.
The first property is immediate, as $\pi^\k$ is continuous and $\cube^\k(Z)$ is closed.
The containment \eqref{eq_cubescontainedinnilcubes} holds for the following reason: if $x \in \cube^\k(X)$ is a limit point of the sequence
\[
n \mapsto (x_n, T^{a(n)}x_n, T^{b(n)} x_n, T^{a(n) + b(n)} x_n)
\]
then $\pi(x)$ is a limit point of $n \mapsto \bigl(\pi(x_n), T^{a(n)}(\pi (x_n)), T^{b(n)}(\pi (x_n)), T^{a(n) + b(n)}(\pi (x_n))\bigr)$.
The calculation
\[
\mu^\k(\nilc^\k(X)) = \mu^\k(( \pi^\k)^{-1} \cube^\k(Z)) = m^\k(\cube^\k(Z)) = 1
\]
gives the last property by \cref{thm:nil_cubes_is_nil}.
\end{proof}

\subsection{Comparing cubes}
\label{subsec:comparing_cubes}

We compare \Erdos{} cubes, nilcubes, and dynamical cubes.
The material of this section is not needed later on, but is included to clarify the main objects studied.

\begin{Definition}
\label{def_erdos_cubes}
For a topological system $(X, T)$, we define the \define{set of $k$-dimensional \Erdos{} cubes} $\erd^\k(X)$ to be the collection of all $k$-dimensional \Erdos{} cubes $x\in X^\k$   (see \cref{def:erdos_cubes}). 
\end{Definition}

\begin{Lemma}
Assume that $(X,T)$ is a topological system.
Then  for all $k\in\N$, we have that $\erd^\k(X) \subset \cube^\k(X)$. 
\end{Lemma}
\begin{proof}
We proceed by induction on $k$.
When $k = 1$, an \Erdos{} cube is of the form $(x,y)$ with $T^{c(n)} x \to y$ as $n \to \infty$ for an increasing sequence $c\colon \N\to\N$.
Since $(x,T^{c(n)} x) \in \cube^\kone(X)$ for every $n \in \N$, it is immediate that $(x,y) \in \cube^\kone(X)$.

For the inductive step, fix an \Erdos{} cube $x \in X^\k$.  
 Recall that $F_1x\in X^\kdown$ is the projection of $x$ onto its first lower face. Since $F_1 x$ is an \Erdos{} cube (see \cref{lem:erdos_face_is_erdos}), we have that $F_1 x \in \cube^\kdown(X)$ and therefore $(F_1 x,(T^\kdown)^n F_1 x) \in \cube^\k(X)$ for all $n \in \N$.
Since  $F^1 x$ is a limit point of $\big\{(T^{\kdown})^nF_1 x:n\in\N\big\}$ and $\cube^\k(X)$ is closed, we are done.
\end{proof}

Note that it does not follow from the  definition that the collection of \Erdos{} cubes is closed, while it is incorporated into the definition of dynamical cubes that this collection is closed. 
Moreover, an \Erdos{} cube is defined in terms of forward orbits under the appropriate cubic transformations, while the definition of a dynamical cube allows both forward and backward orbits.  This difference is motivated by the intended application of the existence of \Erdos{} cubes to find patterns in a set of positive upper density in the positive (and not all) integers.  However, even were we to extend the definition of an \Erdos{} cube to allow two sided orbits, they would still not define the same object as the dynamical cubes.  We include some examples to illustrate these differences.
\begin{example}
Set $X=\{0,1\}^\Z$ endowed with the shift map $T$. 
There exists $x \in X$ such that
\begin{equation}
\label{eqn:not_erdos_cube}
(x, \one_{2\Z}, \one_{3\Z}, \one_{5\Z})
\end{equation}
belongs to $\cube^\two(X)$ but not to $\erd^\two(X)$.
Take $a(n) = 8^{10n}$ and $b(n) = 9^{10n}$.
Let
\[
x(i)=\begin{cases}
1_{2\Z}(i)& \text{ if }i\in[a(n),a(n)+n]\text{ for some }n\\
1_{3\Z}(i)& \text{ if }i\in[b(n),b(n)+n]\text{ for some }n\\
1_{5\Z}(i)& \text{ if }i\in[a(n)+b(n),a(n)+b(n)+n]\text{ for some }n\\
0& \text{otherwise}\\
\end{cases}
\]
Then $(x,1_{2\Z},1_{3\Z},1_{5\Z})$ is a limit of $(x,T^{a(n)}x, T^{b(n)} x, T^{a(n) + b(n)} x)$.
However \eqref{eqn:not_erdos_cube} is not an \Erdos{} cube, as $\one_{5\Z}$ does not lie in the orbit closure of $\one_{2\Z}$.
\end{example}

Even in well-behaved, familiar systems, such as an algebraic skew product, we may not have equality between the two notions of cubes.
\begin{example}
Consider the skew-product $T\colon (x,y)\mapsto(x+\alpha,y+x)$ on $X=\T^2$, where $\alpha$ is irrational.
Then $\big((x_1,y_1),(x_2,y_2),(x_3,y_3),(x_4,y_4)\big)\in X^4$ belongs to $\cube^\two(X)$ if and only if $x_1+x_4=x_2+x_3$ (this can be seen by computing the orbit of a point $(x',y_1)$ with $x'$ close to $x_1$ and rationally independent from $\alpha$). 
However, the point $\big((0,0),(0,0),(0,0),(0,1/2)\big)\in X^4$ is not in $\erd^\two(X)$.
\end{example}

We have the containments
\[
\erd^\k(X) \subset \cube^\k(X) \subset \nilc^\k(X)
\]
for all $k \in \N$ whenever $(X,\mu,T)$ is an ergodic system with topological pronilfactors.
The preceding examples show that the first containment is proper in general, even if we were to adjust the definition to make use of the  full (backwards and forwards) orbit. 
For ergodic nilsystems, and more generally for distal systems, it turns out that the closure of the $k$-dimensional \Erdos{} cubes equals the set of $k$-dimensional dynamical cubes.

\begin{Theorem}
\label{th:erdos-and-dynamical}
Let $(X,T)$ be a distal topological system and let $k\in\N$. 
Then $\cube^\k(X)=\overline{\erd^\k(X)}$.  
\end{Theorem}

\begin{proof}
Define
\[
P = \{ (x,y) :  x \in \cube^\kdown(X), y \in \omg(x,T^\kdown) \}
\]
and note that \cref{lemma_erdos_cubes_symmetries} implies
\[
\erd^\k(X) \supset \bigcap_{i=1}^k \phi_i( P )
\]
where $\phi_i$ is the permutation of $\{1,\dots,k\}$ that exchanges $1$ and $i$.

We use induction to check  that for any $x \in X$ and $n\in\Z^k$, the point $(T^{\epsilon \cdot n} x : \epsilon\in\k)$ is in $P$.
Indeed, when $k=1$, since $(X,T)$ is distal it is a disjoint union of minimal systems by \cite[Theorem 3.2]{furstenberg-distal}.  Restricting to one of these minimal systems, we have that  $(a,T^n a) \in P$ for all $a \in X$ and all $n \in \Z$, as the forward orbit of any $a$ is dense.
The inductive step follows because each power $(X^\k,T^\k)$ of a distal system is also distal.
It follows that $\cube^\k(X) \subset \overline{P}$ and hence  $\cube^\k(X) \subset \overline{\phi_i(P)}$ for any $i\in\{1, \ldots, k\}$, as $\cube^\k(X)$ is invariant under $\phi_i$.  
Therefore
\[
\cube^\k(X)
\subset
\bigcap_{i=1}^k\overline{\phi_i(P)}
\subset
\overline{\erd^\k(X)}
\subset
\cube^\k(X)
\]
and we conclude that $\cube^\k(X)=\overline{\erd^\k(X)}$.
\end{proof}

In general, an ergodic system $(X, \mu, T)$ having topological nilfactors need not have $\cube^\k(X) = \nilc^\k(X)$. Examples illustrating this are given in~\cite[Example 3.6]{TuYe2013} and~\cite[Section 3.2]{CaiShao2019}.

\section{A continuous ergodic decomposition}
\label{sec:cont_erg_decomp}
The rest of the paper is dedicated to the proof of \cref{thm:cubes_exist_pronil}.
The first step
is to generalize \cref{prop_cont_kronecker_implies_cont_erg_decomp} to higher dimensions by describing an ergodic decomposition for the cube systems $(X^\k,\mu^\k,T^\k)$ that is compatible with the topology. 

\begin{Definition}
Given a system $(X,\mu,T)$, we say that a map $x\mapsto\mu_x$ from $X$ to ${\mathcal M}(X)$ is a \emph{continuous ergodic decomposition} if it is continuous and is an ergodic decomposition in the sense of \cref{def_ergdec}.
\end{Definition}
Note that, in view of \cref{thm:erg_decomp}, a continuous ergodic decomposition, if it exists, is unique on $\supp(\mu)$.
However, a continuous ergodic decomposition can be extended from $\supp(\mu)$ to the whole space $X$ in an arbitrary manner (as long as continuity is preserved).
\begin{Lemma}
\label{remark_continuous_ergodic_decomposition}
A system $(X,\mu,T)$ admits a continuous ergodic decomposition if and only if for every $f\in C(X)$, the conditional expectation $\E(f\mid{\mathcal I})$ on the $\sigma$-algebra of $T$-invariant sets agrees $\mu$-almost everywhere with a continuous function.
\end{Lemma}
\begin{proof}
If $x\mapsto\mu_x$ is a continuous ergodic decomposition, then for any $f\in C(X)$ the function 
\[
x\mapsto\int_Xf\d\mu_x
\]
is continuous, and
\[
\int_X f \intd \mu_x = \E(f\mid{\mathcal I})(x)
\]
for $\mu$-almost every $x\in X$, by definition.
Conversely, if for every $f\in C(X)$ there exists $g\in C(X)$ such that $\E(f\mid{\mathcal I})$ and $g$ agree $\mu$-almost everywhere, then the measures $\mu_x$ defined by $\int_Xf\d\mu_x:=g(x)$ define a continuous ergodic decomposition.
\end{proof}

We now give some examples to clarify the notion of a continuous ergodic decomposition.

\begin{example}
Take $X = \R^2 / \Z^2$ and define $T$ by $T(x,y) = (x,y+x)$.
The Haar measure $\mu$ on $X$ is $T$-invariant but is not ergodic.
In this system, every point is generic for some ergodic measure, but the function that associates to each point the ergodic measure it is generic for is not a continuous function. 
To remedy this issue, let $m$ denote the Lebesgue measure on $[0,1)$ and, for each $(x,y)\in X$, define the measure $\mu_{(x,y)}=\delta_x\ m$, where $\delta_x$ denotes as usual the Dirac point mass at $x$.

The map $(x,y) \mapsto \mu_{(x,y)}$ from $X$ to the space $\meas(X)$ of Borel probability measures on $X$ is continuous.
Furthermore, $\mu_{(x,y)}$ is ergodic for $T$ whenever $x$ is irrational, in particular it is ergodic for $\mu$-almost every $(x,y)\in X$.
Finally, since
$\int_X\mu_{(x,y)}\d\mu(x,y)=\mu$, the assignment $(x,y)\mapsto\mu_{(x,y)}$ is a continuous ergodic decomposition of $\mu$.
\end{example}

\begin{example}
For our second example fix $\alpha \in \R \setminus \Q$ and define $T$ on $X = \R^4 / \Z^4$ by
\[
T(x,y,w,z) = (x + \alpha, y + x, w + \alpha, z + w)
\]
for all $(x,y,w,z) \in X$.
The Haar measure $\mu$ on $\T^4$ is $T$-invariant but not ergodic, as the set
\[
X_\beta = \{ (x,y,w,z) \in X : w-x = \beta \}
\]
is $T$-invariant for every $\beta \in \R / \Z$.
Note that $\psi_\beta \colon \R^3 / \Z^3 \to X_\beta$ by
\[
(x,y,z) \mapsto (x, y , x + \beta, z)
\]
defines an isomorphism between $(X_\beta,T)$ and $(\T^3,S_\beta)$, where
\[
S_\beta(x,y,z) = (x + \alpha, y + x, z + x+\beta)
\]
on $\R^3 / \Z^3$.
It can be checked that the Haar measure $m$ on $\T^3$ is ergodic for $S_\beta$ for almost every $\beta\in\R/\Z$.
The map $(x,y,w,z)\mapsto\mu_{(x,y,w,z)} :=  \psi_{w-x}(m)$ assigning to $(x,y,w,z)$ the push forward of $m$ under $\psi_{w-z}$ is a continuous map from $X$ to $\meas(X)$.
Moreover, $\mu_{(x,y,w,z)}$ is $T$-ergodic for $\mu$-almost every $(x,y,w,z)\in X$ and forms a disintegration of $\mu$, showing that $(x,y,w,z)\mapsto\mu_{(x,y,w,z)}$ is a continuous ergodic decomposition of $\mu$.

Note that $T$ is a self-product of the map $(x,y)\mapsto (x+\alpha,y+x)$ on $\R^2 / \Z^2$ and that the Haar measure $\nu$ on $\T^2$ is ergodic for $T$.
This disintegration is therefore an example of a continuous ergodic decomposition of $\nu^\kone = \nu\times\nu$.
\end{example}

Recall that, given a system $(X,\mu,T)$ with topological pronilfactors, we denote by $Z_k$ the $k$-step pronilfactor and by $\pi_k\colon X\to Z_k$ the (continuous) factor map. 
Recall also that $\nilc^\k(X) = (\pi_k^\k)^{-1} (\cube^\k(Z_k))$ is closed and $T^\k$-invariant, and that $\mu^\k(\nilc^\k(X))=1$, so that $(\nilc^\k(X),\mu^\k,T^\k)$ is a system.
In particular, the support of $\mu^\k$ is contained in $\nilc^\k(X)$.

\begin{Theorem}
\label{thm:muk_cont_erg_decomp}
Assume that $(X,\mu, T)$ is an ergodic system with topological pronilfactors.
For every $k \in \N$ there is a unique continuous ergodic decomposition $x\mapsto\lambda^\k_x$
of the system $(\nilc^\k(X),\mu^\k,T^\k)$ such that
\begin{equation}
\label{lambda:equality}
\pi_k^\k(x) = \pi_k^\k(y) \Rightarrow \lambda^\k_x = \lambda^\k_y
\end{equation}
for all $x,y \in \nilc^\k(X)$.
\end{Theorem}

We noted that it follows from \cref{thm:erg_decomp} that continuous ergodic decompositions  are uniquely defined on the support of the measure being decomposed. Yet in the most interesting situations, our point $a$ is outside the support of $\mu$ and therefore all points $x \in X^\k$ whose first coordinate equals $a$ are outside $\supp(\mu^\k)$. 
However, the ergodic decomposition in \cref{thm:muk_cont_erg_decomp} is defined on $\nilc^\k(X)$, which contains points with any prescribed first coordinate (this follows from the fact that our system has topological pronilfactors).

The remainder of this section is devoted to proving \cref{thm:muk_cont_erg_decomp}.  
We start by applying a theorem of Leibman to
establish \cref{thm:muk_cont_erg_decomp} for ergodic nilsystems, and then make use of 
inverse limits and the Host-Kra structure theorem to prove the general case.

\begin{Theorem}[Leibman {\cite[Theorem~2.2]{leibman-orbits}}]
\label{thm:leibman}
Let $(Z,m,T)$ be a nilsystem with $Z = G/\Gamma$ and $\pi\colon G\to Z$ the natural map.
Let $V$ be a connected subnilmanifold of $Z$ and let $K$ be a connected component of $\pi^{-1}(V)$.
There is a subnilmanifold $Y$ of $Z$ with the following property: for $m$-almost every $z \in V$ we have $\overline{ \{ T^n z : n \in \Z \} } = aY$ whenever $\pi(a) = z$ and $a \in K$.
\end{Theorem}

Using \cref{thm:leibman} we deduce that every nilsystem  admits a continuous ergodic decomposition.
\begin{Lemma}
\label{thm:nil_cont_erg_decomp}
Every nilsystem $(Z,m,T)$ admits a continuous ergodic decomposition $\xi \colon Z \to \meas(Z)$.
\end{Lemma}

\begin{proof}
Let us first assume that $Z$ is connected. Once we have established the claim for connected $Z$, we then derive the general case using basic properties of nilsystems.   

Write $Z = G/\Gamma$ where $G$ is a nilpotent Lie group and $\Gamma$ a discrete and co-compact subgroup of $G$, and let $\pi\colon G\to Z$ denote the natural factor map from $G$ onto $Z$.
Let $G^\circ$ denote the identity component of $G$ and note that $G^\circ\Gamma=G$ because $Z=G/\Gamma$ is connected. 
Applying \cref{thm:leibman} with $V=Z$ and $K=G^\circ$ gives a subnilmanifold $Y$ of $Z$ and a set $W \subset G^\circ$ of full measure in $G^\circ$ with the following property: 
for every $g \in W$ we have
\[
\overline{ \{ T^n \pi(g) : n \in \Z \} } = gY.
\]
In particular, we have that $g_1Y=g_2Y$ whenever $g_1,g_2\in W$ with $\pi(g_1)=\pi(g_2)$. 
Define the map $\tilde{\xi}\colon G^\circ\to \meas(Z)$ as $\tilde{\xi}(g)=\mu_{gY}$, where $\mu_{gY}$ denotes the Haar measure of the subnilmanifold $gY$. 
Note that $(gY,T)$ is a transitive nilsystem whenever $g\in W$, and hence $\mu_{gY}$ is ergodic.
Moreover $\pi(g)$ is generic for $\tilde\xi(g)$.
We claim that
\begin{equation}
\label{eqn_xi_inv_1}
\tilde{\xi}(g\gamma)=\tilde{\xi}(g),\qquad\text{ for all } g\in G^\circ, \text{ for all } \gamma\in\Gamma^\circ,
\end{equation}
where $\Gamma^\circ=G^\circ\cap \Gamma$.
By the conclusion of \cref{thm:leibman} it follows that $\tilde{\xi}(g\gamma)=\tilde{\xi}(g)$ for any $g\in W$ and any $\gamma\in \Gamma^\circ$ with $g\gamma\in W$, because 
$\pi(g)=\pi(g\gamma)$ and hence $gY=g\gamma Y$. In other words, we have that \eqref{eqn_xi_inv_1} holds for all $\gamma\in \Gamma^\circ$ and all $g\in W\cap W\gamma^{-1}$. Since $\tilde{\xi}\colon G^\circ\to \meas(Z)$ is a continuous map and $W\cap W\gamma^{-1}$ is a dense subset of $G^\circ$ (because $W\cap W\gamma^{-1}$ has full measure), we conclude that \eqref{eqn_xi_inv_1} holds for all $g\in G^\circ$ and all $\gamma\in \Gamma^\circ$ as desired. Finally, \eqref{eqn_xi_inv_1} implies that the map $\tilde{\xi}$ descents to a map $\xi$ on the quotient $G^\circ/\Gamma^\circ$. Since $G^\circ/\Gamma^\circ$ and $G/\Gamma$ are isomorphic, we can view $\xi$ as a map on $Z=G/\Gamma$ and we are done.

We are left to deal with the case when $Z$ is not connected. In this case, by compactness, the nilmanifold $Z$ has finitely many connected components, which we denote by $Z_1,\dots,Z_{r}$. Let $m_i$ denote the Haar measure on $Z_i$ and observe that $m=\frac{1}{r}(m_1+\cdots+m_{r})$.

Since $T$ maps connected components to connected components, there exists a permutation $\tau\colon\{1,\ldots,r\}\to\{1,\ldots,r\}$ such that $T Z_i=Z_{\tau(i)}$. Since any permutation on $r$ symbols satisfies $\tau^{r!}= \operatorname{Id}$, we have $T^{r!} Z_i=Z_i$ for every $i\in\{1,\ldots,r\}$.
In particular, the triple $(Z_i,m_i, T^{r!})$ is a connected nilsystem.
Hence, there exists a continuous map $\xi_i \colon Z_i \to \meas(Z_i)$ such that $z \mapsto \xi_{i,z}$ is an ergodic decomposition of $m_i$ with respect to the transformation $T^{r!}$, and $m_i$-almost every  $z\in Z_i$ is generic for $\xi_{i,z}$ for all  $i\in\{1,\ldots,r\}$.
Let $U_i$ denote the set of points in $Z_i$ that are generic for $\xi_{i,z}$ (under $T^{r!}$) and let $U=U_1\cup\cdots\cup U_r$; then let $Y\subset Z$ be the set of points whose entire orbit is contained in $U$ and note that $m(Y)=1$.
For $z\in Y\cap Z_i$ we have an explicit description of $\xi_{i,z}$ as
$$\xi_{i,z}=\lim_{N\to\infty}\frac1N\sum_{n=1}^N\delta_{T^{nr!}z}$$
and in particular $T(\xi_{i,z})=\xi_{\tau(i),Tz}$.

For every $z\in Z$, let $i_z$ be the unique number in $\{1,\ldots,r\}$ such that $z\in Z_{i_z}$.
Consider the map $z\mapsto \xi_z$ given by
\[
\xi_z=\frac{1}{r!}\sum_{j=1}^{r!} T^j(\xi_{i_z,z})
\]
and note that it is continuous.
Moreover, each $z\in Y$ is generic for $\xi_z$ under $T$, which immediately implies that $\xi_z$ is $T$-invariant, ergodic, and that $\int_Z\xi_z\d m=m$.
\end{proof}

It now follows that ergodic nilsystems satisfy the conclusion of \cref{thm:muk_cont_erg_decomp}.

\begin{lemma}
\label{lem:nil_muk_cont_erg_decomp}
Let $(Z,m,T)$ be an ergodic nilsystem.
For every $k \in \N$ there is a continuous ergodic decomposition $z\mapsto\xi^\k_z$ of the system $(\cube^\k(Z),m^\k,T^\k)$.
\end{lemma}
\begin{proof}
Fix $k \in \N$.
By \cref{thm:nil_cubes_is_nil}, it follows that $(\cube^\k(Z),T^\k)$ is a nilsystem and that the measure $m^\k$ is equal to the Haar measure on $\cube^\k(Z)$.  Thus the conclusion follows from \cref{thm:nil_cont_erg_decomp}.
\end{proof}

Next, we need a version of \cref{lem:nil_muk_cont_erg_decomp} for topological inverse limits of nilsystems.

\begin{lemma}\label{lemma_InvLimErgDec}
Suppose for each $j\in\N$ we have a system $(X_j,\mu_j,T_j)$ admitting a continuous ergodic decomposition, and suppose there are continuous factor maps $\pi_j\colon X_{j+1}\to X_j$.
Then the inverse limit $(X,\mu,T)$ admits a continuous ergodic decomposition.
\end{lemma}
\begin{proof}
Recall that we have continuous factor maps $\psi_j\colon X\to X_j$  satisfying $\pi_j\circ\psi_j=\psi_{j-1}$.
In view of \cref{remark_continuous_ergodic_decomposition}, for each $j\in\N$ and $f\in C(X_j)$, the function $\E(f\mid{\mathcal I}_j)$ is $\mu_j$-almost everywhere  continuous, and hence the function $\E(f\circ\psi_j\mid{\mathcal I})=\E(f\mid{\mathcal I}_j)\circ\psi_j$ is $\mu$-almost everywhere continuous.
Since
\[
\bigcup_{j\in\N}C(X_j)\circ\psi_j
\]
is dense in $C(X)$, using the converse direction of \cref{remark_continuous_ergodic_decomposition} we conclude that $(X,\mu,T)$ admits a continuous ergodic decomposition.
\end{proof}

\begin{Corollary}
\label{lem:pro_nil_muk_cont_erg_decomp}
Let $k\in\N$ and let $(Z,\mu,T)$ be an inverse limit of ergodic $k$-step nilsystems.
Then the system $(\cube^\k(Z),\mu^\k,T^\k)$ admits a continuous ergodic decomposition.
\end{Corollary}
\begin{proof}
Suppose that $(Z,\mu,T)$ is the inverse limit of a sequence of ergodic, $k$-step nilsystems $(Z_j,m_j,T_j)$.
Then a routine check shows that  $(\cube^\k(Z),\mu^\k,T^\k)$ is the inverse limit of the sequence $(\cube^\k(Z_j),m_j^\k,T_j^\k)$.
The conclusion now follows by combining Lemmas \ref{lem:nil_muk_cont_erg_decomp} and \ref{lemma_InvLimErgDec}.
\end{proof}

\begin{lemma}\label{lemma_ErgodicDecompositionMarginal}
Let $(X,\mu,T)$ and $(Y,\nu,S)$ be systems and suppose $\chi\colon X\to Y$ is a  factor map.
If $\nu$ is ergodic and $x\mapsto\mu_x$ is an ergodic decomposition of $\mu$, then $\chi(\mu_x)$ is equal to $\nu$ for $\mu$-almost every $x\in X$.
\end{lemma}
\begin{proof}
As $\chi$ is a factor map and almost-every $\mu_x$ is $T$-invariant, almost every $\chi(\mu_x)$ is $S$-invariant.
The fact that
\[
\int_X\mu_x\d\mu(x)=\mu
\]
implies
\[
\int_X\chi(\mu_x)\d\mu(x)=\nu
\]
is an ergodic decomposition of $\mu$. Ergodicity of $\nu$ then implies that $\chi(\mu_x)=\nu$ for $\mu$-almost every $x\in X$.
\end{proof}

We have now assembled the tools to prove \cref{thm:muk_cont_erg_decomp}. 

\begin{proof}[Proof of \cref{thm:muk_cont_erg_decomp}]
Fix a system $(X,\mu, T)$ with topological pronilfactors and fix $k \in \N$.
To keep notation simple, and just throughout this proof, we denote by $(Z,m,T)$ the $k$-step pronilfactor of $(X,\mu, T)$, and let $\pi \colon X \to Z$ be the continuous factor map.
We also write $\eta=\pi^\k=\pi_k^\k$ for the map $\eta \colon X^\k \to Z^\k$ that applies $\pi$ coordinatewise.

We begin by noting that if a continuous ergodic decompositions of the system $(\nilc^\k(X),\mu^\k,T^\k)$ satisfying \eqref{lambda:equality} exists, then by \cref{lem:supp_to_cont_magic} applied to the factor map $\eta$, it must be unique.

Apply \cref{lem:pro_nil_muk_cont_erg_decomp} to get a continuous ergodic decomposition $z \mapsto \xi^\k_z$ from $\cube^\k(Z)$ to $\meas(\cube^\k(Z))$.
For each $x\in\nilc^\k(X)$ define a measure $\lambda^\k_x$ on $X^\k$ by
\begin{equation}
\label{eqn:lambda_definition}
\int_{X^\k} \bigotimes_{\epsilon\in\k}f_\epsilon\intd\lambda^\k_x 
=
\int_{Z^\k} \bigotimes_{\epsilon\in\k}\E(f_\epsilon\mid Z) \intd \xi^\k_{\eta(x)}
\end{equation}
whenever $f=(f_\epsilon)_{\epsilon\in\k} \in \cont(X)^\k$.
Since the right hand side of \eqref{eqn:lambda_definition} depends on $x$ only through $\eta(x)$, it is clear that~\eqref{lambda:equality} holds. 

We are left with showing that $x \mapsto \lambda^\k_x$ is a continuous ergodic decomposition.

Write $\inv^\k$ for the $\sigma$-algebra of Borel measurable and $T^\k$ invariant subsets of $X^\k$.

\begin{Claim}\label{claim_1}
For any $f \in \cont(X)^\k$ we have
\begin{equation}
\label{eqn:ergodic_to_nilfactor}
\E \Bigl( \bigotimes_{\epsilon \in \k} f_\epsilon \mid \inv^\k \Bigr)(x)
=
\E \bigg( \bigotimes_{\epsilon \in \k} \E(f_\epsilon\mid Z) \circ \pi \mid \inv^\k \bigg)(x)
\end{equation}
for $\mu^\k$-almost every $x\in X^\k$.
\end{Claim}

To see this, note that the difference between the two sides of \eqref{eqn:ergodic_to_nilfactor} can be written as a sum of $2^k$ terms of the form $\E(\bigotimes_{\epsilon\in\k} g_\epsilon\mid{\mathcal I}^\k)$ where at least one of the $g_\epsilon$ satisfies $\E(g_\epsilon\mid Z)=0$.

Applying the Cauchy-Schwarz-Gowers inequality \cite[Theorem 13, Chapter 8]{HK-book} gives
\begin{align*}
\left\| \E \Bigl( \bigotimes_{\epsilon \in \k} g_\epsilon \mid \inv^\k \Bigr) \right\|^2_{\mu^\k}
& =
\left|
\int_{X^\k}
\E \Bigl( \bigotimes_{\epsilon \in \k} g_\epsilon \mid \inv^\k \Bigr)
\cdot
\E \Bigl( \bigotimes_{\epsilon \in \k} \overline{g_\epsilon} \mid \inv^\k \Bigr)
\intd \mu^\k
\right|
\\
& =
\left|
\int_{X^\kup} \bigotimes_{\epsilon \in \k} g_\epsilon \otimes \bigotimes_{\epsilon \in \k} \overline{g_\epsilon} \intd \mu^\kup
\right|
\le
\prod_{\epsilon \in \k} \ghk g_\epsilon \ghk_{k+1}^2
\end{align*}
and note that $\ghk g \ghk_{k+1} = 0$ if and only if $\condex(g\mid Z) = 0$ by \cref{thm:pronilfactors} because $(Z,m,T)$ is the $k$-step pronilfactor of $(X,\mu,T)$.
This completes the proof of \cref{claim_1}. 

\begin{Claim}\label{claim_3}
Letting  $\inv^\k_Z$ denote the $\sigma$-algebra of $T^\k$-invariant subsets of $\cube^\k(Z)$, we have that 
$$\E \bigg( \bigotimes_{\epsilon \in \k} \E(f_\epsilon\mid Z) \circ \pi \mid \inv^\k \bigg)
=
\E \bigg( \bigotimes_{\epsilon \in \k} \E(f_\epsilon\mid Z)  \mid \inv^\k_Z \bigg)\circ\eta. 
$$
\end{Claim}

To check this, note that for any $g\in L^2(\cube^\k(Z),m^\k)$, the pointwise ergodic theorem and the fact that $\eta$ is a factor map imply that
$$\E(g\mid\inv^\k_Z)\circ\eta
=
\lim_{N\to\infty}\frac1N\sum_{n=1}^N g\circ T^n\circ\eta
=\lim_{N\to\infty}\frac1N\sum_{n=1}^N g\circ\eta\circ T^n
=
\E(g\circ\eta\mid\inv^\k).
$$
If
\[
g=\bigotimes_{\epsilon \in \k} \E(f_\epsilon\mid Z)
\]
then
\[
g\circ\eta=\bigotimes_{\epsilon \in \k} \E(f_\epsilon\mid Z) \circ \pi, 
\]
completing the proof of \cref{claim_3}. 

Since $z\mapsto\xi^\k_z$ is an ergodic decomposition of $m^\k$, combining \cref{claim_3} with \cref{claim_1}, it follows that the definition in \eqref{eqn:lambda_definition} of $\lambda^\k_x$ defines an ergodic decomposition of $\mu^\k$. 

We are left with showing that the mapping $x\mapsto\lambda^\k_x$ is continuous, and to do so it 
suffices to show that for any $f\in C(X)^\k$ the function
\[
x\mapsto\int_{X^\k}\bigotimes_{\epsilon \in \k} f_\epsilon\d\lambda_x^\k
\]
is continuous.

\begin{Claim}
\label{claim2}
For  $p=2^k$, the function $\theta\colon \lp^p(Z,m)^\k\to\lp^\infty(Z,m)$ given by \begin{equation}\label{eq_thetadefined}
    \theta(g)(z)=\int_{Z^\k}\bigotimes_{\epsilon \in \k} g_\epsilon\d\xi_z^\k
\end{equation} is continuous.
\end{Claim}
To see this, since $z\mapsto\xi_z^\k$ is an ergodic decomposition of $m^\k$, \cref{lemma_ErgodicDecompositionMarginal} (applied to the projections $Z^\k\to Z$) together with \cref{lem:supp_to_cont_magic} implies that every projection of $\xi_z^\k$ equals $m$ for every $z\in\cube^\k(Z)$.

Now given $g,\tilde g\in\lp^p(Z,m)^\k$, if there exists $\omega\in\k$ such that $g_\epsilon=\tilde g_\epsilon$ for all $\epsilon\in\k\setminus\{\omega\}$, then by Holder's inequality
\[
\|\theta(g)-\theta(\tilde g)\|_{L^\infty(m)}\leq \prod_{\epsilon\neq\omega}\|g_\epsilon\|_{L^p(m)}\cdot\|g_\omega-\tilde g_\omega\|_{L^p(m)}.
\] 
This implies that $\theta$ is continuous, 
completing the proof of \cref{claim2}. 

Using the definition \eqref{eqn:lambda_definition}, the fact that $\E(f_\epsilon\mid Z)\in\lp^\infty(Z,m)\subset\lp^p(Z,m)$ and that $\eta\colon \nilc^\k(X)\to\cube^\k(Z)$ is continuous, it suffices to show that for any $g\in \lp^p(Z,m)^\k$ the function $\theta(g)$ defined by \eqref{eq_thetadefined} 
is continuous.

Note that for $g\in C(Z)^\k$ we have $\theta(g)\in C(Z)$, since the map $z\mapsto\xi_z^\k$ is continuous.
Therefore, continuity of $\theta$, together with the fact that $C(Z)$ is a closed subset of $\lp^\infty(Z,m)$ and the fact that $C(Z)^\k$ is dense $\lp^p(Z,m)^\k$, together imply that $\theta(g)\in C(Z)$ for any $g\in\lp^p(Z,m)^\k$.
\end{proof}

\begin{Corollary}
\label{cor:lambda_invariance}
Let $(X,\mu, T)$ be an ergodic system with topological pronilfactors, let $k\in\N$ and let $x\mapsto\lambda^\k_x$ be the continuous ergodic decomposition of $(\nilc^\k(X),\mu^\k,T^\k)$ provided by \cref{thm:muk_cont_erg_decomp}.
Then
\[
\lambda_{T^\k x}^\k=\lambda_{x}^\k=T^\k(\lambda_x^\k)
\]
for every $x\in\nilc^\k(X)$.
\end{Corollary}
\begin{proof}
We claim that both assignments $x\mapsto\lambda_{T^\k x}^\k$ and $x\mapsto T^\k(\lambda_x^\k)$ are continuous ergodic decompositions of $(\nilc^\k(X),\mu^\k,T^\k)$ satisfying~\eqref{lambda:equality}. 
In view of the almost everywhere uniqueness of ergodic decompositions we can then invoke \cref{lem:supp_to_cont_magic} to conclude that $\lambda_{T^\k x}^\k=\lambda_{x}^\k=T^\k(\lambda_x^\k)$ for any $x\in\nilc^\k(X)$.

To prove the claim, note that for any continuous function $f\in C(\nilc^\k(X))$ we have
$$\int_{\nilc^\k(X)}f\d\lambda^\k_{T^\k x}=\E(f\mid{\mathcal I}^\k)(T^\k x)=\E(f\mid{\mathcal I}^\k)(x)$$
for $\mu^\k$-almost every $x\in X^\k$, where the last equality follows from the fact that ${\mathcal I}^\k$ is the $\sigma$-algebra of $T^\k$-invariant sets.
This implies that $x\mapsto\lambda_{T^\k x}^\k$ is an ergodic decomposition of $\mu^\k$.
Similarly,
$$\int_{\nilc^\k(X)}f\d T^\k(\lambda_x^\k)=\E(T^\k f\mid{\mathcal I}^\k)(x)=\E(f\mid{\mathcal I}^\k)(x),$$
showing that $x\mapsto T^\k(\lambda_x^\k)$ is also an ergodic decomposition. Continuity follows from the fact that $x\mapsto\lambda_{x}^\k$ is continuous and~\eqref{lambda:equality} follows from the fact that $\pi\colon X\to Z$ is a factor map.
This finishes the proof.
\end{proof}

\begin{Corollary}\label{lambda:marginals}
For every $x\in\nilc^\k(X)$ and every coordinate projection $X^\k\to X$, the push forward of $\lambda^\k_x$ is the measure $\mu$.
\end{Corollary}

In other words, \cref{lambda:marginals} says that for every $x\in\nilc^\k(X)$ the measure $\lambda_x^\k$ is a $2^k$-fold self-joining of $\mu$.

\begin{proof}
By \cref{thm:muk_cont_erg_decomp}, we have that $x \mapsto \lambda^\k_x$ is an ergodic decomposition, and so \cref{lemma_ErgodicDecompositionMarginal} (applied to the projections $X^\k\to X$) gives us the statement for $\mu^\k$-almost every $x\in X^\k$. 
Combining this fact with~\eqref{lambda:equality} and using \cref{lem:supp_to_cont_magic} it follows that the conclusion holds for every $x\in \nilc^\k(X)$.
\end{proof}

\section{Proof of \cref{thm:cubes_exist_pronil}}
%\section{Proof of \texorpdfstring{\cref{thm:cubes_exist_pronil}}{Theorem 5.8}}
\label{sec:proof_cubes_exist_pronil}

Throughout this section, whenever we are given an ergodic system $(X,\mu,T)$ with topological pronilfactors, we use $\lambda_x^\k$ to denote the measures constructed in \cref{thm:muk_cont_erg_decomp}.

\subsection{An explicit disintegration over the first coordinate}
\label{sec:sigmaka}

In this section we use the continuous ergodic decomposition obtained in the previous section to construct a concrete disintegration of $\mu^\k$ with respect to the projection $x \mapsto x_{\vec 0}$ from $X^\k$ to the first coordinate.
The results in this section constitute general versions of \cref{prop_2-dim_sigma_rep} and \cref{thm:sigmaka_symm_k2}.
Recall the map $F^*\colon X^\k\to X^{\k\setminus\{\vec0\}}$ defined in~\eqref{eqn_fstar} that forgets the $\vec 0$ coordinate, in the sense that $x=(x_{\vec0},F^*x)$ for every $x\in X^\k=X\times X^{\k\setminus\{\vec0\}}$.

\begin{Theorem}
\label{thm:sigmaka_exist}
Assume that $(X,\mu, T)$ is an ergodic system with topological pronilfactors.
For every $k\in\N$ and every $t \in X$, there exists a measure $\sigma_t^\k$ on $X^\k$ such that the following hold. 
\begin{enumerate}[label={\textbf{S\arabic*}.},ref={\textbf{S\arabic*}}]
\item
\label{sigma:defined}
For all $t\in X$, we have that $\supp(\sigma_t^\k)\subset\nilc^\k(X)$.
\item
\label{sigma:equations}
The measures satisfy
\begin{equation}
\label{eqn:sigmaka_measures}
\sigma^\kone_t 
=
\delta_t \times \mu; 
\qquad
\sigma^\kup_t 
=
\int_{\nilc^\k(X)} \delta_x \times \lambda^\k_x \intd \sigma^\k_t(x)
\end{equation}
for all $t \in X$.
\item
\label{sigma:cont_disint}
The map $t \mapsto \sigma^\k_t$ is both continuous and a disintegration of the measure $\mu^\k$.
\item
\label{sigma:first_coord}
For all $t \in X$, we have that $\sigma^\k_t( \{ x \in X^\k : x_{\vec 0} = t \} ) = 1$.
\item
\label{sigma:other_coords}
For all $t,s \in X$, whenever $\pi_k(t) = \pi_k(s)$ then we have that $F^* \sigma^\k_t = F^* \sigma^\k_s$.
\end{enumerate}
\end{Theorem}

For the remainder of the section, whenever $(X,\mu,T)$ is an ergodic system with topological pronilfactors, we let $\sigma_t^\k$ denote the measures that are guaranteed to exist by this theorem. 
\begin{proof}

The proof is by induction on $k$.
We begin by defining $\sigma^\kone_t = \delta_t \times \mu$ for all $t \in X$.
It is straightforward to verify $t \mapsto \sigma^\kone_t$ is a continuous map from $X$ to $\meas(X^\kone)$ and a disintegration of $\mu^\kone = \mu \times \mu$.
We also have $x_{0} = t$ for $\sigma^\kone_t$-almost every point and $F^*\sigma^\kone_t= \mu$ for all $t \in X$.
It follows from \cref{lem:transitive_qone_full} that $\nilc^\kone(X)=X^\kone$ for all $t \in X$.
Thus \ref{sigma:defined} through \ref{sigma:other_coords} hold when $k = 1$.

Suppose by induction that for some $k \ge 1$ we have a measure $\sigma^\k_t \in \meas(X^\k)$ for every $t \in X$ satisfying properties \ref{sigma:defined} through \ref{sigma:other_coords}.
We may then define
\[
\sigma^\kup_t 
=
\int_{\nilc^\k(X)} \delta_x \times \lambda^\k_x \intd \sigma^\k_t(x)
\]
for every $t \in X$. This establishes \ref{sigma:equations} for $k+1$, and we are left with verifying the other properties.

Continuity of $x \mapsto \lambda^\k_x$ and $t \mapsto \sigma^\k_t$ gives continuity of $t \mapsto \sigma^\kup_t$.
It is a disintegration of $\mu^\kup$ by the calculation
\begin{align*}
\int_X \sigma^\kup_t \intd \mu(t)
&
=
\int_X\int_{\nilc^\k(X)} \delta_{x} \times \lambda^\k_{x} \intd \sigma^\k_t(x) \intd \mu(t)
\\
&
=
\int_{\nilc^\k(X)} \delta_{x} \times \lambda^\k_{x} \intd \mu^\k(x)
\\
&
=
\mu^\kup
\end{align*}
which uses induction for the second equality and \cref{lem:dont_know_if_obvious} for the third.
This shows that \ref{sigma:cont_disint} holds.

To prove \ref{sigma:first_coord} note that, since $\sigma^\k_t$-almost every point $x$ satisfies $x_{\vec 0} = t$, we can write
\[
\sigma^\kup_t 
=
\delta_t \times \int_{\nilc^\k(X)} \delta_{F^*x} \times \lambda^\k_x \intd \sigma^\k_t(x)
\]
which implies $\sigma^\kup_t$-almost every point $x$ satisfies $x_{\vec 0} = t$.

Next we establish \ref{sigma:other_coords}.
Since $\sigma^\k_t$-almost every point $x\in\nilc^\k(X)$ can be written as $x=(t,F^*x)$, we have
\[
F^* \sigma^\kup_t
=
\int_{\nilc^\k(X)} \delta_{F^*x} \times \lambda^\k_{(t,F^* x)} \intd \sigma^\k_t(x)
=
\int_{F^*(\nilc^\k(X))} \delta_{y} \times \lambda^\k_{(t,y)} \intd (F^* \sigma^\k_t)(y).
\]
Using~\eqref{lambda:equality}, it follows that if $\pi_k(t) = \pi_k(s)$ then $\lambda^\k_{(t,y)}=\lambda^\k_{(s,y)}$ whenever $y\in X^{\k\setminus\{\vec0\}}$ and $(t,y)\in\nilc^\k(X)$.
Using the induction hypothesis it follows that
$F^* \sigma^\kup_t = F^* \sigma^\kup_s$, establishing \ref{sigma:other_coords}.

Finally, we need to verify that \ref{sigma:defined} holds.
Since $\mu^\kup\big(\nilc^\kup(X)\big) = 1$ by \cref{lem_nilcprops}, the fact that $t \mapsto \sigma^\kup_t$ is a disintegration of $\mu^\kup$ gives that $\sigma_t^\k(\nilc^\kup(X))=1$ for $\mu$-almost every $t$.
Now \cref{lem_nilcprops} also states $\nilc^\kup(X)$ is closed so 
\[
\{ \nu \in \meas(X^\kup) : \nu(\nilc^\kup(X)) = 1 \}
\]
is closed. 
Continuity of $t \mapsto \sigma^\kup_t$ then gives $\sigma_t^\k(\nilc^\kup(X))=1$ for every $t \in \supp(\mu)$.
We conclude that the set
\[
\{ t \in X : \sigma^\kup_t(\nilc^\kup(X)) = 1 \}
\]
satisfies the hypothesis of \cref{lem:supp_to_cont_magic} and is therefore equal to $X$.
This establishes \ref{sigma:defined}.
\end{proof}

\begin{Corollary}
\label{cor_lastpointgeneric}
Assume $(X,\mu,T)$ is an ergodic system with topological pronilfactors.
For every $t\in X$ and every coordinate projection $X^\k\to X$ other than projection to $\vec{0}$, the  push forward of $\sigma_t^\k$ equals $\mu$.
\end{Corollary}

\begin{proof}
This follows by combining the  definition \eqref{eqn:sigmaka_measures} of $\sigma_t^\k$ with \cref{lambda:marginals}.
\end{proof}

\begin{Corollary}
\label{cor_permutationinvarianceofsigma}
Assume $(X,\mu,T)$ is an ergodic system with topological pronilfactors.
For every $k\in\N$, every permutation $\phi$ on $\{1,\dots,k\}$, and every $t\in X$ we have $\phi\sigma_t^\k=\sigma_t^\k$ where $\phi\sigma_t^\k$ is  the  push forward of $\sigma_t^\k$ under the map defined in \eqref{eqn:perm_on_measures}.
\end{Corollary}
\begin{proof}
Fix $k\in\N$ and fix a permutation $\phi$ of $k$ digits.
Let $\tilde\sigma_t^\k:=\phi\sigma_t^\k$ for every $t\in X$.
We claim that the measures $\tilde\sigma_t^\k$ satisfy properties \ref{sigma:cont_disint} through \ref{sigma:other_coords} of \cref{thm:sigmaka_exist}.

Recall from \cref{lemma_muk_permutation_invariant} that $\phi\mu^\k=\mu^\k$; therefore $t\mapsto\tilde\sigma_t^\k$ is a disintegration of $\mu^\k$ and we obtain \ref{sigma:cont_disint} for $\tilde\sigma_t^\k$.
Since the set $\{x\in\nilc^\k(X):x_{\vec0}=t\}$ is invariant under $\phi$, we immediately obtain \ref{sigma:first_coord} for $\tilde\sigma_t^\k$ as well.
Finally, since the map $\phi\colon X^\k\to X^\k$ does not involve the first coordinate, we have that $F^*\rho=F^*\tilde\rho$ if and only if $F^*\phi \rho= F^*\phi \tilde\rho$ for any measures $\rho$ and $\tilde\rho$ on $X^\k$.  Thus in particular we deduce that the measures $\tilde\sigma_t^\k$ also satisfy \ref{sigma:other_coords}.

In view of conditions \ref{sigma:cont_disint} and \ref{sigma:first_coord}, both $t\mapsto\sigma_t^\k$ and  $t\mapsto\tilde\sigma_t^\k$ satisfy \cref{thm:disintegration}.
From the uniqueness in that theorem it follows that $\sigma_t^\k=\tilde\sigma_t^\k$ for $\mu$-almost every $t\in X$.
Using continuity from \ref{sigma:cont_disint} and combining it with \ref{sigma:other_coords}, we can now apply \cref{lem:supp_to_cont_magic} to conclude that $\sigma_t^\k=\tilde\sigma_t^\k$ for all $t\in X$.
\end{proof}

We also need equivariance of the disintegration $t \mapsto \sigma^\k_t$ with respect to  $T^\k$.

\begin{Corollary}
\label{lem:sigma_equivariance}
Assume $(X,\mu,T)$ is an ergodic system with topological pronilfactors.
For every $k \in \N$ and every $t \in X$, we have $T^\k \sigma^\k_t = \sigma^\k_{T(t)}$.
\end{Corollary}
\begin{proof}
The case $k = 1$ is immediate as
\[
T^\kone \sigma^\kone_t = T^\kone (\delta_t \times \mu) = \delta_{T(t)} \times \mu = \sigma^\kone_{T(t)}
\]
for all $t \in X$.
The proof in general is by induction, using the calculation
\begin{align*}
T^\kup \sigma^\kup_t
&
=
\int_{\nilc^\kup(X)} T^\k \delta_x \times T^\k \lambda^\k_x \intd \sigma^\k_t(x)
\\
&
=
\int_{\nilc^\kup(X)} \delta_{T^\k(x)} \times \lambda^\k_{T^\k (x)} \intd \sigma^\k_t(x)
\\
&
=
\int_{\nilc^\kup(X)} \delta_{x} \times \lambda^\k_{x} \intd (T^\k \sigma^\k_t)(x)
=
\sigma^\kup_{T(t)}
\end{align*}
which employs \cref{cor:lambda_invariance}.
\end{proof}

\subsection{The Existence of \Erdos{} Cubes}
\label{sec:cubes_exist}

We start by reformulating \cref{thm:cubes_exist_pronil} in terms of the measure $\sigma^\k_a$.

\begin{Theorem}
\label{thm:sigmas_live_on_cubes}
Assume that $(X,\mu, T)$ is an ergodic system with topological pronilfactors and that $a \in \gen(\mu,\Phi)$ for some \Folner{} sequence $\Phi$.
For every $k \in \N$, $\sigma^\k_a$-almost every $x \in \nilc^\k(X) $ is a $k$-dimensional \Erdos{} cube.
\end{Theorem}

\begin{proof}[Proof that \cref{thm:sigmas_live_on_cubes} implies \cref{thm:cubes_exist_pronil}]
Fix an ergodic system $(X,\mu,T)$ with topological pronilfactors, let $a \in \gen(\mu,\Phi)$ for some \Folner{} sequence $\Phi$, let $E \subset X$ with $\mu(E) > 0$ and let $k \in \N$.
By \cref{thm:sigmas_live_on_cubes}, we have that that  $\sigma^\k_a$-almost every $x \in \nilc^\k(X)$ is a $k$-dimensional \Erdos{} cube.
Since $\sigma^\k_a$-almost every point satisfies $x_{\vec 0} = a$ by \ref{sigma:first_coord} and $\sigma^\k_a$ pushes forward to $\mu$ on the last coordinate by \cref{cor_lastpointgeneric}, we can in particular choose a $k$-dimensional \Erdos{} cube with $x_{\vec 0} = a$ and $x_{\vec 1} \in E$.
\end{proof}

The remainder of this section is devoted to the proof of  \cref{thm:sigmas_live_on_cubes}, and for this we need one last ingredient: that $\sigma^\k_a$-almost every point $x$ is generic for $\lambda^\k_x$ along some \Folner{} sequence. 
This is the only place in the proof of \cref{thm:cubes_exist} where the assumption $a\in\gen(\mu,\Phi)$ is used.

\begin{Theorem}
\label{thm:sigmaka_generics}
Assume that $(X,\mu, T)$ is an ergodic system with topological pronilfactors and that $a \in \gen(\mu,\Phi)$ for some \Folner{} sequence $\Phi$.
For each $k \in \N$, there exists a \Folner{} sequence $\Psi$ such that
\[
\sigma^\k_a ( \{ x \in \nilc^\k(X)  : x \in \gen(\lambda^\k_x,\Psi) \}) = 1
\]
holds.
\end{Theorem}

\begin{proof} 
Using \cref{lem:cont_erg_generics} and Property \ref{sigma:cont_disint} of~\cref{thm:sigmaka_exist}, we get
\begin{equation}
\label{eq:full-measure}
\sigma_b^\k( \{ x \in \nilc^\k(X)  : x \in \gen(\lambda^\k_{x}) \}) = 1.
\end{equation}
for $\mu$-almost every $b\in X$.
In particular, we can fix $b\in\supp(\mu)$ satisfying \eqref{eq:full-measure}.
By \cref{lem:gen_approximates_support}, we can choose a sequence $(s(n))_{n\in\N}$ of integers such that $T^{s(n)} a \to b$ as $n \to \infty$. 
As the map $t \mapsto \sigma_t^\k$ is continuous, it follows that we  have the convergence of the measures 
\begin{equation}
\label{eq:measures-converge}
(T^\k)^{s(n)} \sigma_a^\k \to \sigma_b^\k \quad\text{ as } n\to\infty
\end{equation}
via \cref{lem:sigma_equivariance}.

Fix a sequence of functions $G_1,G_2,\ldots\in \cont(X^\k)$ that is dense in $\cont(X^\k)$ with respect to the supremum norm.
From \eqref{eq:full-measure} the set of points $x \in \nilc^\k(X)$ that are generic for $\lambda^\k_x$ has full  
$\sigma_b^\k$ measure.
By the monotone convergence theorem, for every $m\in\N$ there exists $N(m) \in\N$ such that the set
\[
A_m^{(1)}
=
\biggl\{ x\in \nilc^\k(X)  :
\max_{1\leq j\leq m} \biggl|\frac{1}{N(m)}\sum_{n=1}^{N(m)} G_j(  (T^\k)^n x)
- \int_{X^\k} G_j \intd\lambda^\k_{x} \biggr| \leq \dfrac{1}{m} \biggr\}
\]
satisfies $\sigma_b^\k(A_m^{(1)}) \geq 1-1/2^m$.  

Let $\varepsilon_m>0$ and take $A_m^{(2)}$ to be the set of 
points in $X^\k$ whose distance from some point in $A_k^{(1)}$ is strictly less than $\varepsilon_m$. 
Choosing $\varepsilon_m$ sufficiently small, it follows that 
\[
\max_{1\leq j\leq m} \left|\frac{1}{N(m)}\sum_{n =1}^{N(m)} G_j\bigl((T^\k)^nx\bigr) - \int_{X^\k} G_j \intd\lambda^\k_{x}\right|\leq \dfrac{2}{m} 
\]
for all $x\in A_m^{(2)}$.
Since $A_m^{(1)}$ is closed and $A_m^{(2)}$ is open (both as subsets of $X^\k$), by Urysohn's lemma there exists a continuous function $\phi_m\colon X^\k\to[0,1]$ satisfying 
\[
\phi_m(x) = 
\begin{cases} 
1 & \text{  for all } x\in A_m^{(1)} \\ 
0 & \text{ for all } x\notin A_m^{(2)}.
\end{cases}
\]
Passing to a subsequence of $(s(n))_{n \in \N}$ if needed and using the convergence of the measures in~\eqref{eq:measures-converge}, we can further assume that
\[
\left|\int_{X^\k} (T^\k)^{s(n)} \phi_m\intd \sigma_a^\k - \int_{X^\k}  \phi_m\intd \sigma_b^\k \right|\leq \frac{1}{2^m}. 
\]
It follows that  
\begin{align*}
\sigma_a^\k\bigl((T^\k)^{-s(m)} A_m^{(2)} \bigr) & \geq \int_{X^\k} (T^\k)^{s(m)} \phi_m \intd \sigma_a^\k 
\geq
\int_{X^\k}  \phi_m\intd \sigma_b^\k - \frac{1}{2^m}
\\
&
\geq \sigma_b^\k(A_m^{(1)})- \frac{1}{2^m} \geq 1-\frac{1}{2^{m-1}}.
\end{align*}

Let $Y$ denote the set of all points in the support of $\sigma_a^\k$  
that belong to all but finitely many $(T^\k)^{-s(m)} A_m^{(2)}$, $m\in\N$. 
By the Borel-Cantelli lemma, $Y$ has full $\sigma_a^\k$-measure.  
For any $x\in Y$ and cofinitely many $m\in\N$, we have
\[
\max_{1\leq j\leq m} \biggl|\frac{1}{N(m)} \sum_{n =1}^{N(m)}  G_j \bigl((T^\k)^{n+s(m)}x\bigr)
- \int_{X^\k} G_j \intd\lambda^\k_{(T^\k)^{s(m)}x}\biggr|\leq \frac{3}{m}.
\]
Since
$
\lambda^\k_{(T^\k)^{s(m)}x}=\lambda^\k_{x}$ by \cref{cor:lambda_invariance},
the conclusion follows with the \Folner{} sequence $\Psi_m = \{ s(m), s(m) + 1,\dots, s(m) + N(m) \}$.
\end{proof}

We can now prove \cref{thm:sigmas_live_on_cubes}. In the language of \cref{subsec:comparing_cubes}, we show that $\sigma^\k_a(\erd^\k(X)) = 1$.
Although this may appear similar to \ref{sigma:defined} in \cref{thm:sigmaka_exist}, we cannot prove it so simply because the set $\erd^\k(X)$ is in general not closed.

\begin{proof}[Proof of \cref{thm:sigmas_live_on_cubes}]
We must show that $\sigma^\k_a$-almost every point $x$ is a $k$-dimensional \Erdos{} cube.
In view of \cref{lemma_erdos_cubes_symmetries} and \cref{cor_permutationinvarianceofsigma}, it suffices to show that for $\sigma_a^\k$-almost every  $(x,y)\in\nilc^\k(X)$ we have $y\in\omega(x,T^\kdown)$.
Using \cref{lem:gen_approximates_support}, the proof is finished if we show that  $\sigma_a^\k$-almost every  $(x,y)\in\nilc^\k(X)$ satisfies both $y\in\supp(\lambda_x^\kdown)$ and $x\in\gen(\lambda_x^\kdown,\Psi)$ for a \Folner{} sequence $\Psi$.

Since $\lambda^\kdown_x\big(\supp(\lambda_x^\kdown)\big)=1$ for every $x\in\nilc^{\kdown}$ it follows that the set
\[
A_1 = \big\{(w,y) \in \nilc^\k(X) : y \in \supp(\lambda_w^\kdown)\big\}
\]
satisfies $\big(\delta_x\times \lambda_x^\kdown\big)(A_1)=1$ for every $x\in\nilc^{\kdown}(X)$.
From the construction in \eqref{eqn:sigmaka_measures} of $\sigma_a^\k$ we conclude that $\sigma_a^\k(A_1)=1$, and hence that $\sigma_a^\k$-almost every  $(x,y)\in\nilc^\k(X)$ satisfies $y\in\supp(\lambda_x^\kdown)$.

Finally, we apply \cref{thm:sigmaka_generics} (with $k-1$ in the place of $k$) to find a \Folner{} sequence $\Psi$ such that the set
\[
A_2 = \big\{x \in \nilc^\kdown(X)  : x \in \gen(\lambda^\kdown_x,\Psi)\big\}
\]
has $\sigma_a^\kdown(A_2)=1$.
Using again \eqref{eqn:sigmaka_measures} we conclude that $\sigma_a^\k(A_2\times X^\kdown)=1$, and hence that $\sigma_a^\k$-almost every  $(x,y)\in\nilc^\k(X)$ satisfies $x\in\gen(\lambda_x^\kdown,\Psi)$.
This concludes the proof.
\end{proof}

\section{Open Questions}
\label{sec:questions}

\subsection{Combinatorics}

Our first question is a natural extension of our main result.
\begin{Question}
Given $A\subset\N$ with positive upper Banach density, are there infinite sets $B_1,B_2,B_3,\dots$ such that for every $k\in\N$ we have $B_1+\cdots+B_k\subset A$? 
\end{Question}
We note that the same question with each $B_i$ being finite (even if arbitrarily large) is easy: simply iterate the fact that whenever $d(A)>0$, for any $m\in\N$ we can find integers $n_1,\dots,n_m$ such that $d\big((A-n_1)\cap\cdots\cap (A-n_m)\big)>0$.

On the other hand, it is important that we only consider the initial sums $B_1+\cdots+B_k$, as parity issues prevent more general sums, such as $B_i+\cdots+B_k$, to being all contained in $A$ (this can be readily seen by taking $A$ to be the set of odd numbers and each $B_j$ to be a singleton).
While this simple example illustrates this phenomenon convincingly, one can also fabricate examples which are aperiodic, and in general it is not clear what are the exact obstructions.
The answer may depend on the precise pattern we are trying to find, so we formulate only a simple case of this general question.

\begin{Question}
Which sets $A\subset\N$ with $d_\Phi(A)>0$ for some F\o lner sequence $\Phi$ contain $B_1\cup B_2\cup(B_1+B_2)$ for some infinite sets $B_1,B_2\subset\N$?
\end{Question}

For example, does it suffice that $1_A-d_\Phi(A)$ is a weak mixing function along $\Phi$ in the sense of \cite[Definition 3.17]{MRR}?

Under the assumption of the Hardy-Littlewood $k$-tuplets prime conjecture, Granville \cite{Granville90} shows that the primes contain $B_1+B_2$ for infinite sets $B_1,B_2\subset\N$, but an unconditional proof is still unknown (see \cite[Question 6.4]{MRR}).
Keeping in mind the heuristic that the primes behave pseudo-randomly, we record the following.

\begin{Question}
Given $A\subset\N$, which notions of pseudo-randomness suffice for $A$ to contain $B_1+B_2$ for infinite sets $B_1,B_2\subset\N$?
\end{Question}

More generally, which other notions of largeness on a set $A$ imply that it contains a sumset $B+C$?

The next question is inspired by the main result from \cite{Moreira}.
\begin{Question}
Given a finite partition $\N=A_1\cup\cdots\cup A_r$, are there infinite sets $B_1,B_2\subset\N$ and a color $A\in\{A_1,\dots,A_r\}$ such that $(B_1+B_2)\cup(B_1B_2)\subset A$?
\end{Question}

This question was also asked by \Erdos{} in \cite[Page 58]{erdos-1977}.
An affirmative answer to the significantly weaker question where $|B_1|=|B_2|=1$ is given in \cite{Moreira}.
The case $|B_1|=2$ and $|B_2|=1$ also follows from the results in that paper, but already the case $|B_1|=|B_2|=2$ seems to be out of reach.

\subsection{Dynamics}
\cref{thm:nil_cont_erg_decomp} asserts that any nilsystems possesses a continuous ergodic decomposition. This naturally leads to the following question. 
\begin{Question}
\label{q_cont_erg_decomp_distal}
What other classes of systems always possesses a continuous ergodic decomposition?
\end{Question}
A natural class to consider, which extends the class of nilsystems, are distal systems. However, for distal systems the answer to \cref{q_cont_erg_decomp_distal} turns out to be negative, as Furstenberg~\cite{Furstenberg61}  exhibited a skew-product $(x,y)\mapsto (x+\alpha,y+h(x))$ which is topologically minimal but not uniquely ergodic.
Taking a non-ergodic measure on such system, there can be no continuous ergodic decomposition, since the only continuous invariant functions are constant, contradicting \cref{remark_continuous_ergodic_decomposition}. 

Another natural class to consider, which extends the class of nilsystems, is the class of horocycle flows on homogeneous spaces. 
Given their algebraic origin, it is possible that a continuous ergodic decomposition always exists.

\subsection{Amenable Groups}
\label{subsec:amenable}

The proof given in~\cite{MRR} for the two-fold sumset holds in an arbitrary amenable group. 
While our proof of \cref{thm:main_theorem} does not hold in this generality due to the reliance on the structure theory for cube systems, the new proof of the case $k=2$ in \cref{sec_B+C} does, and may lead to the answers of some questions in this direction that have previously been out of reach.

Recall that a countable group $G$ is \define{amenable} if there is a sequence $(\Phi_N)_{N \in \N}$ of finite subsets of $G$ such that
\begin{gather}
\label{eqn:r_folner}
\lim_{N \to \infty} \dfrac{|\Phi_N g \cap \Phi_N|}{|\Phi_N|}
=
1
\\
\label{eqn:l_folner}
\lim_{N \to \infty} \dfrac{|\Phi_N \cap g \Phi_N|}{|\Phi_N|}
=
1
\end{gather}
both hold for all $g \in G$.
Any such sequence is called a \define{two-sided \Folner{} sequence} on $G$.
In \cite[Theorem~1.3]{MRR} it was proved that if $A\subset G$ satisfies
\begin{equation}
\label{eqn:amenable_pos_dens}
\limsup_{N \to \infty} \dfrac{|A \cap \Phi_N|}{|\Phi_N|} > 0
\end{equation}
for a two-sided \Folner{} sequence $(\Phi_N)_{N \in \N}$, then $A$ contains the product $BC = \{ bc : b \in B, c \in C \}$ of two infinite sets $B,C \subset G$.
It is natural to ask whether the same is true for $k$-fold products.

\begin{Question}
Let $G$ be an amenable group, let $k\in\N$ and let $A\subset G$ satisfy \eqref{eqn:amenable_pos_dens} with respect to some two-sided \Folner{} sequence.
Do there always exist infinite sets $B_1,\dots,B_k\subset G$ such that $B_1\cdots B_k\subset A$?
\end{Question}

In non-commutative groups we are interested in two-sided versions of the above question, and which of the assumptions \eqref{eqn:r_folner}, \eqref{eqn:l_folner} are necessary for various types of product sets to be found in all positive density sets.

\begin{Question}
Let $G$ be an amenable group and let $A\subset G$ satisfy \eqref{eqn:amenable_pos_dens} with respect to some two-sided \Folner{} sequence.
Does $A$ necessarily contain $BC \cup CB$ for infinite sets $B,C\subset G$?
\end{Question}

\begin{Question}
Let $G$ be an amenable group.
Which of the assumptions \eqref{eqn:r_folner}, \eqref{eqn:l_folner} on a sequence $(\Phi_N)_{N \in \Phi}$ of finite subsets of $G$ guarantee every $A \subset G$ satisfying \eqref{eqn:amenable_pos_dens} contains a product set $BC$ for infinite sets $B,C\subset G$?
\end{Question}

\subsection{Ultrafilters}
Ultrafilters on the natural numbers were a key tool in~\cite{MRR} and so we enquire here about a strengthening of \cref{thm:main_theorem_k2} in those terms.
One can use ultrafilters on $\N$ to characterize when a set $A \subset \N$ contains $B+C$ with $B,C \subset \N$ infinite in terms of the associative operation
\[
\ultra{p} + \ultra{q} = \{ A \subset \N : \{ n \in \N : A - n \in \ultra{q} \} \in \ultra{p} \}
\]
on the set of all ultrafilters. Indeed, the set $A \subset \N$ contains a sumset if and only if one can find non-principal ultrafilters $\ultra{p}$ and $\ultra{q}$ with $A$ belonging to both $\ultra{p} + \ultra{q}$ and $\ultra{q} + \ultra{p}$.
We remark that, using ultrafilter limits, the notion of an \Erdos{} cube arises naturally from this characterization.

\begin{Question}
Is it true that for any $A\subset \N$ with positive upper Banach density there exist non-principal ultrafilters $\ultra{p}$ and $\ultra{q}$ with $\ultra{p} + \ultra{q} = \ultra{q} + \ultra{p}$ containing $A$?
\end{Question}

One could generalize this question to ask about a strengthening of \cref{thm:main_theorem} in terms of ultrafilters, and even further to generalizations of the questions in \cref{subsec:amenable}.

\bibliographystyle{amsplain}
\bibliography{bcd.bib}

\end{document}